\documentclass[12pt,a4paper,psamsfonts]{amsart}
\usepackage{amssymb,amscd,amsxtra,calc}
\usepackage{cmmib57}
\usepackage[dvips,a4paper,bookmarks,bookmarksopen,%
bookmarksnumbered,pdfauthor={Nakayama and Zhang},%
colorlinks=false]{hyperref}

\theoremstyle{plain}
    \newtheorem{theorem}{Theorem}[section] %
   \renewcommand{\thetheorem}%
    {\arabic{section}.\arabic{theorem}}
    \newtheorem{corollary}[theorem]{Corollary}
    \newtheorem*{ccorollary}{Corollary}
    \newtheorem{proposition}[theorem]{Proposition}
    \newtheorem{lemma}[theorem]{Lemma}
\theoremstyle{definition}
    \newtheorem{dfn}[theorem]{Definition}
    \newtheorem{conjecture}[theorem]{Conjecture}
\theoremstyle{remark}
    \newtheorem*{cclaim}{Claim}
    \newtheorem{rremark}[theorem]{Remark}
    \newtheorem*{remark}{Remark}
    \newtheorem{example}[theorem]{Example}
\newcommand{\BCC}{\mathbb{C}}
\newcommand{\BPP}{\mathbb{P}}
\newcommand{\BQQ}{\mathbb{Q}}
\newcommand{\BRR}{\mathbb{R}}
\newcommand{\BZZ}{\mathbb{Z}}

\newcommand{\SE}{\mathcal{E}}
\newcommand{\SI}{\mathcal{I}}
\newcommand{\SJ}{\mathcal{J}}

\newcommand{\SO}{\mathcal{O}}

\newcommand{\ScS}{\mathcal{S}}
\newcommand{\SU}{\mathcal{U}}
\newcommand{\SV}{\mathcal{V}}
\newcommand{\SY}{\mathcal{Y}}

\newcommand{\NN}{\mathsf{N}}

\newcommand{\Xtt}{\mathtt{X}}
\newcommand{\Ztt}{\mathtt{Z}}

\newcommand{\Aut}{\operatorname{Aut}}
\newcommand{\Bir}{\operatorname{Bir}}

\newcommand{\Sur}{\operatorname{Sur}}

\newcommand{\Codim}{\operatorname{codim}}

\newcommand{\Alb}{\operatorname{Alb}} %Albanese variety
\newcommand{\alg}{\operatorname{alg}} %algebraic pi one
\newcommand{\Gal}{\operatorname{Gal}}
\newcommand{\SHom}{\operatorname{\mathcal{H}\mathit{om}}}
\newcommand{\Imm}{\operatorname{Im}}
\newcommand{\OH}{\operatorname{H}} %%%Homology, Cohomology group
\newcommand{\OR}{\operatorname{R}} %%%Higher direct image
\newcommand{\NS}{\operatorname{NS}}
\newcommand{\Sing}{\operatorname{Sing}}
\newcommand{\Supp}{\operatorname{Supp}}

\newcommand{\mult}{\operatorname{mult}}

\newcommand{\isom}{\simeq} %%isomorphism
\newcommand{\ratmap}%%rational map
{{\,\cdot\negmedspace\cdot\negmedspace\cdot\negmedspace\to\,}}
\newcommand{\injmap}{\hookrightarrow}

\newcommand{\q}{\mathsf{q}}

\newcommand{\reg}{\mathrm{reg}}

\newcommand{\numeqpre}{%
\overset{\raise.01ex\hbox{$\overset%
{\raise.01ex\hbox{$\sim$}}{\smash{\sim}}$}}{\smash{\sim}}
}
\newcommand{\numeqsymb}{%
\lower.25ex\hbox{$\numeqpre$}}
\newcommand{\numeq}{\mathrel{\numeqsymb}}

\begin{document}
%\begin{large}
\title[]%[Polarized endomorphisms]%
{Polarized endomorphisms of complex normal varieties}
\author[]{Noboru Nakayama}
\address%[Noboru Nakayama]
{\textsc{Research Institute for Mathematical Sciences} \endgraf
\textsc{Kyoto University, Kyoto 606-8502, Japan}}
\email{nakayama@kurims.kyoto-u.ac.jp}

\author[]{De-Qi Zhang}
\address%[De-Qi Zhang]
{\textsc{Department of Mathematics} \endgraf
\textsc{National University of Singapore, 2 Science Drive 2,
Singapore 117543, Singapore}}
\email{matzdq@nus.edu.sg}

%\date{\today}

\begin{abstract}
It is shown that a complex normal projective variety
has non-positive Kodaira dimension if it admits a non-isomorphic
quasi-polarized
endomorphism. The geometric structure of the variety is described by
methods of equivariant lifting and fibrations.
%Endomorphisms of the projective spaces are also discussed and some
%results are obtained for hypersurfaces invariant
%under the pullback of the endomorphism.
\end{abstract}

\subjclass[2000]{14J10, 14E20, 32H50}
\keywords{endomorphism, Calabi-Yau variety, rationally connected variety}
\maketitle

\raggedbottom

\section{Introduction}

We work over the complex number field \( \BCC \).
Much progress has been recently made in the study of
endomorphisms of smooth projective varieties from the algebro-geometric
viewpoint.
Especially, the following cases of varieties are well studied:
projective surfaces (\cite{Ny02}, \cite{FN1}),
homogeneous manifolds (\cite{PS}, \cite{Ct}),
Fano manifolds (\cite{Am97}, \cite{ARV}, \cite{HM}),
projective bundles (\cite{Am}), and
projective threefolds with non-negative Kodaira dimension
(\cite{Fm}, \cite{FN2}).
Additionally, \'etale endomorphisms are investigated in \cite{NZ}
from the viewpoint of the birational classification of algebraic
varieties.
However, there is neither a classification of endomorphisms of
singular varieties even when they are of dimension two (except for \cite{ENS}),
nor any reasonably fine classification of
non-\'etale endomorphisms of smooth threefolds,
which are then necessarily uniruled.

Let \( V \) be a normal projective variety of dimension \( n \).
An endomorphism \( f \colon V \to V \) is called \emph{polarized}
if there is an ample divisor \( H \) such that \( f^{*}H \) is
linearly equivalent to \( \q H \) (\( f^{*}H \sim \q H \))
for a positive number \( \q \).
In this case, \( f \) is a finite surjective morphism,
\( \q \) is an integer, and \( \deg f = \q^{n} \)
(cf.\ Lemma~\ref{lem:first} below).
A surjective endomorphism of a variety of Picard number one is always
polarized.
Polarized endomorphisms of smooth projective varieties are
studied in papers \cite{Fa} and \cite{Zs}.
In the present article, we shall study the polarized endomorphisms of
normal projective varieties (not only smooth ones).
The following
Theorems~\ref{Th0} and \ref{Th1.2new}
%and \ref{ThE}
are our main results.

\begin{theorem}\label{Th0}
Let \( f \colon X \to X \) be
a non-isomorphic polarized endomorphism of a normal projective variety
\( X \). Then there exist a finite morphism \( \tau \colon V \to X \)
from a normal projective variety \( V \),
a dominant rational map \( \pi \colon V \ratmap A \times S \) for an
abelian variety \( A \) and a weak Calabi--Yau variety \(S\)
\emph{(cf.\ Definition~\ref{dfn:wCY} below)}, and
polarized endomorphisms \( f_{V} \colon V \to V \), \( f_{A} \colon A
\to A \) and \( f_{S} \colon S \to S \) satisfying the following
conditions\emph{:}
\begin{enumerate}
\item  \( \tau \circ f_{V} = f \circ \tau\),
\( \pi \circ f_{V} = (f_{A} \times f_{S}) \circ \pi  \),
i.e., the diagram below is commutative\emph{:}
\newcommand{\lratmap}{%
\longleftarrow\negmedspace\cdot\negmedspace\cdot\negmedspace\cdot}
\[ \begin{array}{ccccc}
A \times S & \overset{\pi}{\lratmap} & V & \overset{\tau}{\longrightarrow} & X \\
\scriptstyle{f_{A} \times f_{S}} \Big\downarrow \phantom{\scriptstyle{f_{A}}} &
&  \scriptstyle{f_{V}}\, \Big\downarrow \, \phantom{\scriptstyle{f_{V}}}
& & \scriptstyle{f} \Big\downarrow \phantom{\scriptstyle{f}} \\
A \times S & \overset{\pi}{\lratmap} & V
& \overset{\tau}{\longrightarrow} & \phantom{.}X.
\end{array}\]

\item  \( \tau \) is \'etale in codimension one.

\item If \( X \) is not uniruled, then the Kodaira dimension \( \kappa(X) = 0 \)
and \( \pi \) is an isomorphism.

\item \label{Th0:4} If \( X \) is uniruled, then,
for the graph \( \Gamma_{\pi}  \subset V \times A \times S \)
of \( \pi \),
the projection \( \Gamma_{\pi}  \to A \times S\) is an
equi-dimensional morphism birational to the
maximal rationally connected fibration
\emph{(MRC fibration in the sense of \cite{KoMM}, cf.\
\cite{Cp92}, \cite{GHS})}
of a smooth model of \( V \).

\item If \( \dim S > 0 \), then \( \dim S \geq 4 \) and \( S \)
contains a non-quotient singular point.
\end{enumerate}
\end{theorem}

In case \( X \) is smooth and \( \kappa(X) \geq 0 \),
Theorem~\ref{Th0} with \( S \) being a point
is proved in \cite{Fa}, Theorem~4.2.
For uniruled \( X \), there is a discussion related to Theorem~\ref{Th0}
on endomorphisms and maximal rationally connected fibrations in \cite{Zs},
Section 2.2, especially in Proposition~2.2.4
(cf.\ Remark~\ref{rem:DelignePair} below).
We expect that \( \dim S = 0 \) for the variety \( S \) in Theorem~\ref{Th0}.
To be precise, we propose:

\begin{conjecture}\label{conj:Qab}
A non-uniruled normal projective variety admitting
a non-isomorphic polarized endomorphism is Q-abelian
(cf.\ Definition~\ref{dfn:Qab} below),
i.e., there is a finite surjective morphism \'etale in codimension one
from an abelian variety onto the variety.
\end{conjecture}

The conjecture has been proved affirmatively
in \cite{Fa}, Theorem~4.2, for the case of
smooth varieties with non-negative Kodaira dimension.
In Theorem~\ref{thm:solQab} below, we confirm the conjecture
for a non-uniruled variety \( X \) such that
\( \dim X \leq 3 \) or that \( X \) has only quotient singularities.

Applying Theorem~\ref{Th0} and more, we have the following
classification result,
where \(q^{\natural}(X, f) \) denotes the supremum
of irregularities \( q(\widetilde{X}') \) of a smooth model \( \widetilde{X}' \)
of \( X' \) for all the finite coverings
\( \tau \colon X' \to X \) \'etale in codimension one and admitting an endomorphism
\( f' \colon X' \to X' \) with \( \tau \circ f' = f \circ \tau \);
we also define a similar notion \( q^{\natural}(X) \) (independent of \( f \))
so that \(q^{\natural}(X, f) \le q^{\natural}(X) \) in general, with equality holds
when \( X \) is non-uniruled (cf.\ Definition~\ref{dfn:qnat}, Proposition~\ref{prop:qXfqX},
Theorem~\ref{thm:nonuniruled}); see also Lemmas~\ref{lem:qYqX} and \ref{lem:circsharp1}.

\begin{theorem}\label{Th1.2new}
Let \( f \colon X \to X \) be a non-isomorphic polarized endomorphism of a normal
projective variety \( X \) of dimension \( n \).
Then \( \kappa(X) \leq 0 \) and \( q^{\natural}(X, f) \leq n\).
Furthermore, \( X \) is described as follows\emph{:}
\begin{enumerate}
\item \label{Th1.2new:0}
Assume that \( q^{\natural}(X, f) = 0 \).
If \( n \leq 3 \), or more generally, if \emph{Conjecture~\ref{conj:Qab}} is true
for varieties of dimension at most \( n \), then
\( X \) is rationally connected.
\item \label{Th1.2new:n}
\( q^{\natural}(X, f) = n \) if and only if \( X \) is Q-abelian
\emph{(cf.\ Definition~\ref{dfn:Qab} below)}.

\item \label{Th1.2new:other}
Assume that \( q^{\natural}(X, f) \geq n - 3 \), or more generally,
that \emph{Conjecture~\ref{conj:Qab}}
is true for varieties of dimension at most \( n - q^{\natural}(X, f) \).
Then there exist
a finite covering \( \tau \colon V \to X \)
\'etale in codimension one,
a birational morphism \( \rho \colon Z \to V \)
of normal projective varieties, and
a flat surjective morphism \( \varpi \colon Z \to A \)
onto an abelian variety \( A \) of dimension \( q^{\natural}(X, f) \),
and polarized endomorphisms \( f_{V} \colon V \to V \),
\( f_{Z} \colon Z \to Z \), \( f_{A} \colon A \to A \)
such that
\begin{itemize}
\item  every fiber of \( \varpi \) is irreducible, normal, and rationally connected,
\item %%%added
\( \tau \circ f_{V} = f \circ \tau\), \( \rho \circ f_{Z} = f_{V} \circ \rho \),
and \( \varpi \circ f_{Z} = f_{A} \circ \varpi\), i.e.,
the diagram below is commutative\emph{:}
\[ \begin{CD}
A @<{\varpi}<< Z @>{\rho}>> V @>{\tau}>> X \\
@V{f_{A}}VV @V{f_{Z}}VV @V{f_{V}}VV @V{f}VV \\
A @<{\varpi}<< Z @>{\rho}>> V @>{\tau}>> \phantom{.}X.
\end{CD} \]
\end{itemize}
Moreover, the fundamental group \( \pi_{1}(X) \) has a
finite-index subgroup which is a finitely generated abelian group of
rank at most \( 2q^{\natural}(X, f) \).

\item  \label{Th1.2new:n-1} If \( q^{\natural}(X, f) = n - 1 \), then there is
a finite covering \( \tau \colon V \to X \) %%%added the name ``\tau''
\'etale in codimension one from a normal projective variety \( V \)
admitting an endomorphism \( f_{V} \colon V \to V \) %%%added the endomorphism
with \( \tau \circ f_{V} = f \circ \tau \)
such that one of the following conditions is satisfied\emph{:}
\begin{enumerate}
\item \label{Th1.2new:n-1:a}
\( V \) is a \( \BPP^{1} \)-bundle over an abelian variety.

\item  \label{Th1.2new:n-1:b}
There exist a \( \BPP^{1} \)-bundle \( Z \) over an abelian
variety and a birational morphism \( Z \to V \) whose exceptional locus
is a section of the \( \BPP^{1} \)-bundle.
\end{enumerate}
\end{enumerate}

\end{theorem}

\subsection*{Notation and Conventions}

The readers may refer to the standard references such as
\cite{KMM} and \cite{KM} for things related to
the birational classification theory of algebraic
varieties and the minimal model theory of projective varieties,
e.g., log-terminal singularity, canonical singularity, etc.

For a normal variety \( X \), the \emph{canonical divisor}, denoted by \( K_{X} \),
is defined as the natural extension of the canonical divisor of the
smooth locus of \( X \).
Details on a relation between the canonical divisor and the dualizing sheaf
\( \omega_{X} \) and details on Weil divisors on normal varieties are explained in
\cite{Re80}, Appendix to \S 1. The notion of \emph{canonical singularity}
is introduced in the same paper \cite{Re80}.

The \emph{Kodaira dimension} \( \kappa(M) \)
of a smooth projective variety \( M \) is a birational invariant.
The Kodaira dimension \( \kappa(X) \) of a singular projective variety \( X \)
is defined as the Kodaira dimension \( \kappa(M) \) of
a \emph{smooth model} \( M \) of \( X \),
i.e., a smooth projective variety \( M \) birational to \( X \).

The \emph{linear equivalence} relation of divisors
is denoted by the symbol \( \sim \),
the \( \BQQ \)-\emph{linear equivalence} relation by
\( \sim_{\BQQ} \), and the \emph{numerical equivalence}
relation by \( \numeq \).

For a projective variety \(Z\), the \emph{singular locus} is denoted by
\( \Sing Z \) and the \emph{smooth locus} \( Z \setminus \Sing Z \) by \( Z_{\reg} \).

Let \( f \colon Z' \to Z \) be a finite surjective morphism of normal varieties.
We denote by \( R_{f} \) the \emph{ramification divisor} of \( f \), which is just
the natural extension of the ramification divisor of the restriction
\( Z'_{\reg} \cap f^{-1}(Z_{\reg}) \to Z_{\reg} \) of \( f \), where
the closed subset \( Z' \setminus (Z'_{\reg} \cap f^{-1}(Z_{\reg})) \)
has codimension at least two;
in other words, for a prime divisor \( \Gamma \) on \( Z' \),
\( \mult_{\Gamma}(R_{f}) = m - 1 \) if and only
if \( \mult_{\Gamma} \left(f^{*}(f(\Gamma))\right) = m\),
where \( \mult_{\Gamma} (D)\) denotes the multiplicity
of a divisor \( D \) along \( \Gamma \).
As usual, we have the \emph{ramification formula}:
\( K_{Z'} = f^{*}(K_{Z}) + R_{f} \).
The finite surjective morphism is called \emph{\'etale in codimension one}
if \( R_{f} = 0 \) or equivalently if \( f \) is \'etale over
\( Z \setminus \Sigma \) for a closed subset \( \Sigma \)
with \( \Codim (\Sigma) \geq 2  \).

\begin{rremark}\label{etale1}
If \( Z \) is smooth and if \( \Sigma \) is a closed subset
with \( \Codim (\Sigma) \geq 2\),
then the natural homomorphism
\( \pi_{1}(Z \setminus \Sigma) \to \pi_{1}(Z) \) of the fundamental groups
is isomorphic.
This property implies the birational invariance of the fundamental group of
a smooth projective variety.
Moreover, the same property implies that if \( Z \) is smooth and
if \( Z' \to Z \) is a finite surjective
morphism \'etale in codimension one from a normal variety \( Z' \),
then, it is actually \'etale.
Therefore, for an arbitrary normal variety \( V \),
a finite surjective morphism \( V' \to V \) \'etale in codimension one
from a normal variety \( V' \) is determined uniquely
up to isomorphism over \( V \)
by a finite index subgroup of \( \pi_{1}(V_{\reg}) \).
\end{rremark}

The \emph{irregularity} \( q(X) \) of a normal projective variety \( X \) is
defined as \( \dim \OH^{1}(X, \SO_{X}) \).
In Definition~\ref{dfn:qnat} below,
we define \( q^{\circ}(X) \) to be the supremum of \( q(X') \) for
all the normal projective varieties \( X' \)
with finite surjective morphisms \( X' \to X \)
\'etale in codimension one.
More variants of the irregularity \( q(X) \)
are defined in Definition~\ref{dfn:qnat}.

A normal projective variety \(Y\) with only canonical singularities
is called a \emph{weak Calabi--Yau variety}
if \(K_Y \sim 0\) and \(q^{\circ}(Y) = 0\)
(cf.\ Definition~\ref{dfn:wCY} below and its remark).

A normal projective variety \(W\) is called \emph{Q-abelian}
if there is a finite surjective morphism
\(A \to W\) \'etale in codimension one from an abelian variety \(A\)
(cf.\ Definition~\ref{dfn:Qab} below and its remark).
A similar notion ``Q-torus'' is introduced in \cite{Ny99},
which is a K\"ahler version and is restricted to \'etale coverings.

An endomorphism \( f \colon X \to X \) is called \emph{polarized}
(resp.\ \emph{quasi-polarized}) if
\( f^{*}H \sim \q H \) for an ample divisor
(resp.\ a nef and big divisor) \( H \) for some positive integer \( \q \) (cf.\
Lemma~\ref{lem:first} below).

\subsection*{Acknowledgement}
The second author would like to express his gratitude to
Professors Fr\'ed\'eric Campana and
Nessim Sibony for the valuable comments, and
to Research Institute for Mathematical Sciences, Kyoto University
for the support and warm hospitality during the visit in the second half of 2007.
He would also like to thank
the following institutes for the support and hospitality:
University of Tokyo, Nagoya University, and Osaka University.

\section{Some basic properties}

A surjective endomorphism of a normal projective variety is a finite
morphism by the same argument as in \cite{Fm}, Lemma~2.3. In fact,
such an endomorphism \( f \colon X \to X \) induces an automorphism
\( f^{*} \colon \NN^{1}(X) \to \NN^{1}(X) \) of the real vector
space \( \NN^{1}(X) := \NS(X) \otimes \BRR \)
for the N\'eron--Severi group \( \NS(X) \), so the
pullback of an ample divisor is ample,
which implies the finiteness of \( f\).

\begin{lemma}\label{lem:first}
Let \( f \colon X \to X \) be an endomorphism of an
\( n \)-dimensional normal projective
variety \( X \) such that \( f^{*}H \numeq \q H \) for a positive
number \( \q \) and for a nef and big divisor \( H \).
Then  \( \q \) is a positive integer and \( \deg f = \q^{n} \).
Moreover, the absolute value of any eigenvalue of \( f^{*} \colon
\NN^{1}(X) \to \NN^{1}(X) \) is \( \q \).
\end{lemma}

\begin{proof}
Comparing the self-intersection numbers \( (f^{*}H)^{n} \) and
\( H^{n}\), we have \( \deg f = \q^{n} \).
In particular, \( \q \) is an algebraic integer.
Since the numerical equivalence classes of \( f^{*}H \) and \( H \)
in \( \NN^{1}(X) \) belong to the image of N\'eron--Severi group \( \NS(X) \),
\( \q \) is a rational number.
Hence, \( \q \) is an integer.
Let \( \lambda \) be the spectral radius of \( f^{*} \colon \NN^{1}(X)
\to \NN^{1}(X)\), i.e.,
the maximum of the absolute values of eigenvalues of \( f^{*} \).
Then there is a nef \( \BRR \)-Cartier \( \BRR \)-divisor \( D \)
such that \( D \not\numeq 0 \) and
\( f^{*}D \numeq \lambda D \), by a version of
the Perron--Frobenius theorem (cf.\ \cite{Bi}).
Suppose that \( \lambda \ne \q \). Then \( DH^{n-1} = 0 \) by
the equalities
\[ \lambda \q^{n-1} DH^{n-1} =
f^{*}D (f^{*}H)^{n-1} = (\deg f)DH^{n-1} = \q^{n}DH^{n-1}. \]
Thus, \( D \numeq 0 \) by Lemma~\ref{lem:HI} below.
This is a contradiction. Therefore, \( \lambda = \q \).
For the spectral radius  \( \lambda' \) of \( (f^{*})^{-1} \),
\( \lambda^{\prime -1} \) is the minimum of the absolute values of
eigenvalues of \( f^{*} \).
% By \cite{Bi},
By the same version of the Perron--Frobenius theorem,
we also have a nef \( \BRR \)-Cartier \( \BRR \)-divisor
\( D' \) such that \( D' \not\numeq 0\) and
\( f^{*}D' = \lambda^{\prime -1}D'\).
Then, \( \lambda' = \q^{-1} \) by the same reason as above.
Hence, the absolute value of any eigenvalue of \( f^{*} \) is \( \q \).
\end{proof}

\begin{ccorollary}
The degree of a quasi-polarized endomorphism of
a normal projective variety of dimension \( n \) is the \( n \)-th power
\( \q^{n} \) of a positive integer \( \q \).
% Let \( f \) be a quasi-polarized endomorphism of
% a normal projective variety of dimension \( n \).
% Then, \( \deg f = \q^{n} \) for a positive integer \( \q \).
\end{ccorollary}

The lemma below is regarded as a generalization of a part of
the Hodge index theorem and
it is used in the proofs of Lemma~\ref{lem:first} and
Theorem~\ref{thm:nonuniruled}.

\begin{lemma}\label{lem:HI}
Let \( X \) be a smooth projective variety of dimension \( n \geq 2\).
Suppose that an \( \BRR \)-divisor \( D \) satisfies
the following two conditions\emph{:}
\begin{enumerate}
\item \label{lem:HI:1} \( DGL_{1}\ldots L_{n - 2} \geq 0 \)
for any effective \( \BRR \)-divisor \( G \)
and for any nef \( \BRR \)-divisors \( L_{1} \), \ldots, \( L_{n - 2} \).

\item \label{lem:HI:2} \( DH_{1}H_{2}\cdots H_{n - 1} = 0 \) for
some nef and big \( \BRR \)-divisors \( H_{1} \), \ldots, \( H_{n-1} \).
\end{enumerate}
Then \( D \) is numerically trivial.
\end{lemma}

\begin{proof}
Let \( A  \) be an ample divisor on \( X \).
Then, there exist a rational number \( a \) and
an effective \( \BRR \)-divisor \( E \) such that \( aH_{1} \numeq E + A \),
since \( H_{1} \) is big.
Thus,
\[ 0 \leq DAH_{2}\dots H_{n - 1} = - D E H_{2}\ldots H_{n - 1} \leq 0\]
by \eqref{lem:HI:1} and \eqref{lem:HI:2}.
Hence, we may assume \( H_{1} = A \).
Applying the same argument to \( H_{i} \) for \( i \geq 2 \),
we have \( DA^{n-1} = 0 \).
Hence, \( D^{2}A^{n - 2} \leq 0 \) in which the equality holds if and only if
\( D \numeq 0 \), by the hard Lefschetz theorem.
Thus, it suffices to show: \( D^{2}A^{n - 2} \geq 0 \).
There is a positive integer \( b \) such that \( D + bA \) is ample.
In particular, \( D + bA \numeq \Delta \)
for an effective \( \BRR \)-divisor \( \Delta \). Hence, we have
\( D^{2}A^{n - 2} \geq 0 \) by \eqref{lem:HI:1}, since
\[ 0 \leq D \Delta A^{n - 2} = D (D + b A) A^{n - 2} = D^{2}A^{n - 2}. \qedhere\]
\end{proof}

The endomorphism in Lemma~\ref{lem:first} is shown to be
quasi-polarized by the following:

\begin{lemma}\label{lem:qpol}
Let \( f \colon X \to X \) be an endomorphism of an
\( n \)-dimensional normal projective
variety \( X \) such that \( f^{*}H \numeq \q H \) for a positive
number \( \q \) and for a nef and big divisor \( H \).
Then the following assertions hold\emph{:}
\begin{enumerate}
\item \label{lem:qpol:1}
The absolute value of any eigenvalue of
\( f^{*} \colon \OH^{1}(X, \SO_{X}) \to \OH^{1}(X, \SO_{X}) \) is
\( \sqrt{\q} \).

\item \label{lem:qpol:2}
There is a nef and big divisor \( H' \) such that
\( H' \numeq H \) and \( f^{*}H' \sim_{\BQQ} \q H' \).
\end{enumerate}
In particular, \( f \) is quasi-polarized by \( H' \).
\end{lemma}

\begin{proof}
\eqref{lem:qpol:1}:
There exist birational morphisms
\( \mu \colon M \to X \) and \( \nu \colon Z \to X \)
from smooth projective varieties \( M \) and \( Z \), and a generically
finite surjective morphism \( h \colon Z \to M \) such that \( \mu
\circ h = f \circ \nu\).
We may assume that the birational map
\( \psi := \mu^{-1} \circ \nu \colon Z \to M \) is holomorphic.
Then, we have a commutative diagram:
\[ \begin{CD}
\OH^{1}(X, \SO_{X}) @>{f^{*}}>> \OH^{1}(X, \SO_{X}) @= \OH^{1}(X, \SO_{X})\\
@V{\mu^{*}}VV @V{\nu^{*}}VV @V{\mu^{*}}VV\\
\OH^{1}(M, \SO_{M}) @>{h^{*}}>> \OH^{1}(Z, \SO_{Z})
@<{\psi^{*}}<{\isom}< \phantom{.}\OH^{1}(M, \SO_{M}).
\end{CD} \]
Let \( \phi(x) \) be the image of \( x \in \OH^{1}(X, \SO_{X}) \)
by the composition
\[ \OH^{1}(X, \SO_{X}) \xrightarrow{\mu^{*}} \OH^{1}(M, \SO_{M})
\xrightarrow{\isom} \OH^{0, 1}(M) \subset \OH^{1}(M, \BCC), \]
where \( \OH^{0, 1}(M) \) is the \( (0, 1) \)-part of the Hodge
decomposition of \( \OH^{1}(M, \BCC) \).
Then, for \( x \in \OH^{1}(X, \SO_{X}) \), we have
\( \psi^{*}\phi(f^{*}(x)) = h^{*}\phi(x) \) by the diagram above.
We consider the following Hermitian form on
\( \OH^{1}(X, \SO_{X}) \):
\[ \langle x, y \rangle  =
-\sqrt{-1} \int_{M} \phi(x) \cup \overline{\phi(y)} \cup
(\mu^{*}c_{1}(H))^{n - 1} \in \BCC.\]
This is positive definite by Lemma~\ref{lem:posdef} below
applied to \( L = \mu^{*}(H) \).
We have the equality
\[ \langle f^{*}(x), f^{*}(y) \rangle = \q \langle x, y \rangle \]
for \( x \), \( y \in \OH^{1}(X, \SO_{X}) \)
by the calculation
\begin{align*}
(\deg h) \langle x, y \rangle &=
-\sqrt{-1}\int_{Z} h^{*}(\phi(x)) \cup \overline{h^{*}(\phi(y))}
\cup (h^{*}\mu^{*}c_{1}(H))^{n - 1} \\
&= -\sqrt{-1} \int_{Z} \psi^{*}\phi(f^{*}(x))
\cup \overline{\psi^{*}\phi(f^{*}(y))} \cup
(\nu^{*}f^{*}c_{1}(H))^{n - 1} \\
&= -\sqrt{-1} \int_{M} \phi(f^{*}(x)) \cup \overline{\phi(f^{*}(y))}
\cup (\mu^{*}f^{*}c_{1}(H))^{n-1} \\
&= \q^{n-1}\langle f^{*}(x), f^{*}(y) \rangle,
\end{align*}
where \( \deg h = \deg f = \q^{n} \).
Therefore, \( \q^{-1/2}f^{*} \) is a unitary transformation with
respect to \( \langle \phantom{x}, \phantom{y} \rangle \).
Thus, the absolute value of any eigenvalue of \( \q^{-1/2}f^{*} \)
is \( 1 \).

\eqref{lem:qpol:2}:
Let \( m \) be the order of \( c_{1}(f^{*}H - \q H) \) in
\( \OH^{2}(X, \BZZ) \).
By the exponential exact sequence
\[ \OH^{1}(X, \SO_{X}) \xrightarrow{\epsilon} \OH^{1}(X, \SO_{X}^{\star}) \to
\OH^{2}(X, \BZZ) \]
we can find an element \( x \in \OH^{1}(X, \SO_{X}) \) with
\( \SO_{X}(m (f^{*}H - \q H)) = m\epsilon(x) \).
There is an element \( y \in \OH^{1}(X, \SO_{X}) \) such that
\( f^{*}(y) - \q y = x \) by \eqref{lem:qpol:1}.
Let \( H' \) be a divisor such that \( \SO_{X}(H - H') = \epsilon(y) \).
Then \( m(f^{*}H' - \q H') \sim 0  \).
Thus, we are done.
\end{proof}

\begin{remark}
The proof of Lemma~\ref{lem:qpol} is similar to
that of \cite{Zs}, Theorem~1.1.2, where \( X \) is assumed
to be smooth.
\end{remark}

In the proof of Lemma~\ref{lem:qpol}, we used the result below:

\begin{lemma}\label{lem:posdef}
Let \( M \) be an \( n \)-dimensional smooth projective variety and
\( L \) a nef and big divisor. Then the Hermitian form
\( \langle \phantom{x}, \phantom{y} \rangle \) on \( \OH^{0, 1}(M) \)
defined by
\[ \langle \xi, \eta \rangle = -\sqrt{-1} \int_{M} \xi \cup
\overline{\eta} \cup c_{1}(L)^{n-1}\]
is positive definite.
\end{lemma}

\begin{proof}
We may assume that \( n \geq 2 \), since it is well-known to be
positive definite in case \( n = 1 \).
If \( L \) is ample, then the bilinear form is positive definite by the hard
Lefschetz theorem.
Thus, the bilinear form is positive semi-definite even if we replace \(
L \) with a nef divisor.
Let \( W \) be a prime divisor of \( M \). Then
\[ -\sqrt{-1}\int_{M} \xi \cup \overline{\xi} \cup c_{1}(L)^{n-2}
\cup c_{1}(W) \]
is non-negative for any \( \xi \in \OH^{0, 1}(M) \).
In fact, it is equal to
\[ -\sqrt{-1}\int_{\widetilde{W}} \varphi^{*}(\xi) \cup
\overline{\varphi^{*}(\xi)} \cup c_{1}(\varphi^{*}L)^{n-2} \]
for a resolution of singularities \( \varphi \colon \widetilde{W} \to W
\), and it is non-negative by the reason above.

There exist a positive integer \( m \),
a smooth ample divisor \( A \), and an effective divisor
\( E = \sum e_{i} E_{i}\) such that \( m L \sim A + E\).
Then
\begin{align*}
m\langle \xi, \xi \rangle &= -\sqrt{-1}\int_{M} \xi \cup \overline{\xi} \cup
m c_{1}(L)^{n-1} \\
&= -\sqrt{-1} \int_{A} \xi|_{A} \cup \overline{\xi|_{A}} \cup
c_{1}(L|_{A})^{n-2}
+ \sum e_{i}(-\sqrt{-1}) \int_{E_{i}} \xi|_{E_{i}}
\cup \overline{\xi|_{E_{i}}} \cup
c_{1}(L|_{E_{i}})^{n-2}.
\end{align*}
Hence, if \( \langle \xi, \xi \rangle = 0 \), then
\[ -\sqrt{-1} \int_{A} \xi|_{A} \cup \overline{\xi|_{A}} \cup
c_{1}(L|_{A})^{n-2} = 0. \]
Since \( L|_{A} \) is nef and big, we can consider the
induction on \( \dim M \). Then, we have
\( \xi|_{A}  = 0 \) as an element of \(\OH^{0, 1}(A) \)
by induction.
Therefore, \( \xi = 0 \), since \( \OH^{1}(M, \SO_{M}(-A)) = 0 \) by
the Kodaira vanishing theorem for \( n \geq 2 \) and
hence \( \OH^{1}(M, \SO_{M}) \to \OH^{1}(A, \SO_{A}) \) is injective.
Thus, we are done.
\end{proof}

We borrow the following property of Galois closures of powers \( f^{k} =
f \circ \cdots \circ f \) from \cite{ENS}:

\begin{lemma}\label{lem:Galcl}
Let \( f \colon X \to X \) be a non-isomorphic surjective endomorphism
of a normal projective variety \( X \).
Let \( \theta_{k} \colon V_{k} \to X \) be the Galois closure of
\( f^{k} \colon X \to X \) for \( k \geq 1 \) and let \( \tau_{k} \colon
V_{k} \to X \) be the induced finite Galois covering such that
\( \theta_{k} = f^{k} \circ \tau_{k} \). Then there exist finite
Galois morphisms \( g_{k}, h_{k} \colon V_{k+1} \to V_{k} \)  such that
\( \tau_{k} \circ g_{k} = \tau_{k+1}  \) and
\( \tau_{k} \circ h_{k} = f \circ \tau_{k+1}\).
\end{lemma}

\begin{proof}
The composite \( f^{k} \circ \tau_{k+1} \colon V_{k+1} \to X \to X \)
is Galois, since so is \( f^{k+1} \circ \tau_{k+1} = \theta_{k+1} \).
Hence, \( f^{k} \circ \tau_{k+1} \) factors through the Galois closure
\( \theta_{k} \) of \( f^{k} \).
Thus, \( \tau_{k+1} = \tau_{k} \circ g_{k}\) for a morphism
\( g_{k} \colon V_{k+1} \to V_{k} \).
Let \( H_{i} \) be the
Galois group of \( f^{i} \circ \tau_{k+1} \colon V_{k+1} \to X\) for \(
0 \leq i \leq k+1 \). Then \( V_{k} \) is regarded as the Galois
closure of \( V_{k+1}/H_{1} \to V_{k+1}/H_{k+1} \), thus \( V_{k} \isom
V_{k+1}/H \) for the maximal normal subgroup \( H \) of \( H_{k+1} \)
contained in \( H_{1} \). Hence, we have a morphism \( h_{k} \colon
V_{k+1} \to V_{k} \) with
\( \tau_{k} \circ h_{k} = f \circ \tau_{k+1}\).
\end{proof}

\begin{dfn}\label{dfn:qnat}

Let \( X \) be a normal projective variety.
The irregularity \( q(X) \) is defined as \( \dim \OH^{1}(X, \SO_{X}) \).
We define the following variants
of \( q(X) \):

\begin{enumerate}
\item \( \tilde{q}(X) := q(\widetilde{X}) \) for a smooth
model \( \widetilde{X} \) of \( X \)
(This is well-defined).

\item \( q^{\circ}(X) \) (resp.\ \( q^{\natural}(X) \) )
is defined to be the supremum of \( q(X') \) (resp.\ \( \tilde{q}(X') \) )
for a normal projective variety \( X' \)
with a finite surjective morphism \( \tau \colon X' \to X \)
\'etale in codimension one. Namely,
\begin{align*}
q^{\circ}(X) &:= \sup
\{q(X') \mid X' \to X \text{ is finite, surjective,
and \'etale in codimension one} \}, \\
q^{\natural}(X) &:=
\sup \{\tilde{q}(X') \mid X' \to X \text{ is finite, surjective, and
\'etale in codimension one} \}.
\end{align*}

\item
Suppose that \( X \) admits a surjective
endomorphism \( f \colon X \to X \).
Then we define \( q^{\circ}(X, f)  \) (resp.\ \( q^{\natural}(X, f) \) ) to be the
supremum of \( q(X') \) (resp.\ \(\tilde{q}(X')\) )
for a normal projective variety \( X' \)
with a finite surjective morphism \( \tau \colon X' \to X \)
\'etale in codimension one and with an endomorphism \( f' \colon X' \to X' \)
such that  \( \tau \circ f' = f \circ \tau\).
\end{enumerate}
\end{dfn}

\begin{remark}
If \( X \) is a smooth projective variety,
then \( q^{\circ}(X) \) equals \( q^{\text{max}}(X) \) defined in \cite{NZ}.
\end{remark}

\begin{rremark}\label{rem:logterm}
Let \( \tau \colon X' \to X \)
be a finite surjective morphism of normal varieties \'etale in
codimension one. Then:
\begin{enumerate}
\item  \label{rem:logterm:1}
\( X \) has only log-terminal
singularities if and only if so does \( X' \).

\item \label{rem:logterm:2}
If \( X \) has only canonical singularities, then so does \( X' \).
\end{enumerate}
These well-known properties are derived from
\cite{Re80}, Proposition~(1.7) and
\cite{Ka}, Proposition~1.7, as follows.
The assertion \eqref{rem:logterm:1} is proved just by the same argument as
in the proof of \cite{Ka}, Proposition~1.7.
If we replace the logarithmic ramification formula
in the proof with the usual ramification formula, then
we can prove \eqref{rem:logterm:2};
this was already done in \cite{Re80}, Proposition~(1.7), (I).
Another proof of these properties is found in \cite{KM}, Proposition~5.20, but
it is essentially the same as above.
\end{rremark}

\begin{lemma}\label{lem:circsharp1}
Let \( X \) be a normal projective variety with
only log-terminal singularities. Then
\( q^{\circ}(X) = q^{\natural}(X) \).
If \( f \) is a surjective endomorphism of \( X \), then
\( q^{\circ}(X) \geq q^{\circ}(X, f) = q^{\natural}(X, f) \).
\end{lemma}

\begin{proof}
If \( X' \) is a normal variety with
a finite covering \( X' \to X \) \'etale in codimension one,
then \( X' \) is also log-terminal
(cf.\ Remark~\ref{rem:logterm}).
In particular, \( X' \) has only rational singularities,
and hence \( q(X') = \tilde{q}(X') \). Thus,
\( q^{\circ}(X) = q^{\natural}(X) \).
Considering the special case where \( X' \) admits
an endomorphism compatible with \( f \),
we have \( q^{\circ}(X, f) = q^{\natural}(X, f) \).
We also have \( q^{\circ}(X) \geq q^{\circ}(X, f) \) by definition.
\end{proof}

\begin{dfn}\label{dfn:wCY}
A normal projective variety \(Y\) with only canonical singularities
is called a \emph{weak Calabi--Yau variety}
if \(K_Y \sim 0\) and \(q^{\circ}(Y) = 0\).
\end{dfn}

\begin{remark}
The notion of weak Calabi--Yau variety
is slightly different from that in \cite{NZ} in which only
finite \'etale coverings were taken into consideration.
A weak Calabi--Yau variety has dimension at least two.
A two-dimensional weak Calabi--Yau variety
is nothing but a normal projective surface such that the minimal resolution of
singularities is a K3 surface and that there is no finite surjective
morphism from any abelian surface.
\end{remark}

\begin{proposition}\label{prop:split}
Let \( X \) be a normal projective variety with only log-terminal singularities
such that \( K_{X} \sim_{\BQQ} 0 \). Then\emph{:}
\begin{enumerate}
\item \label{prop:split:1} \( q^{\circ}(X) \leq \dim X\).
In particular, there is a finite Galois covering \( X' \to X \)
\'etale in codimension one such that \( q(X') = q^{\circ}(X) \).

\item \label{prop:split:2}
\( q(X) = \dim X \) if and only if \( X \) is an abelian variety.

\item \label{prop:split:3}
There exists a finite covering \( A \times S \to X \)
\'etale in codimension one for an abelian variety \( A \) of dimension \( q^{\circ}(X) \)
and a weak Calabi--Yau variety \( S \).

\end{enumerate}
\end{proposition}

\begin{proof}
Let \( r \) be the smallest positive integer such that \( rK_{X} \sim 0 \).
Then, there is a cyclic covering \( \widehat{X} \to X \)
of degree \( r \) \'etale in codimension one
from a normal projective variety \( \widehat{X} \)
such that \( K_{\widehat{X}} \sim 0 \).
The covering is unique up to isomorphism over \( X \)
and is called the\emph{ global index-one covering}
(or the \emph{canonical cover} in \cite{Ka}).
Then, \( \widehat{X} \) has only canonical singularities by \cite{Ka}, Proposition~1.7.
Let \( Y \to \widehat{X} \) be a finite covering \'etale in codimension one from
a normal projective variety \( Y \). Then, \( K_{Y} \sim 0 \)
and \( Y \) has only canonical singularities by \cite{Re80}, Proposition~(1.7).
Let \( \alpha \colon Y \to A := \Alb(Y) \) be the Albanese map;
this is holomorphic, since \( Y \) has only rational singularities
(cf.\ \cite{Ka85}, Lemma~8.1).
Then, \( \alpha \) is an \'etale fiber bundle by \cite{Ka85}, Theorem~8.3, i.e.,
there is a finite \'etale covering \( A' \to A \) from an abelian variety \( A' \)
such that \( Y \times_{A} A' \isom F \times A' \) over \( A' \)
for a fiber \( F \) of \( \alpha \).
In particular, \( q(Y) = \dim A \leq \dim X \).
As a consequence, we have \( q^{\circ}(X) \leq \dim X \),
since any finite covering \( X' \to X \) \'etale in codimension one is dominated by
such a variety \( Y \).
By the boundedness of \( q^{\circ}(X) \),
we have a finite covering \( X' \to X \) \'etale in codimension one such that
\( q^{\circ}(X) = q(X') \). The Galois closure \( X'' \to X \) of \( X' \to X \) is
also \'etale in codimension one and
\( q^{\circ}(X) = q(X') \leq q(X'') \leq q^{\circ}(X)  \).
Thus, \eqref{prop:split:1} has been proved.

In order to prove the other assertions \eqref{prop:split:2} and \eqref{prop:split:3},
we may assume that \( q^{\circ}(X) = q(Y) \) and that
the composite \( Y \to \widehat{X} \to X \) is Galois.
Let \( G \) be the Galois group of
\( Y \to X \).

Assume that \( q(X) = \dim X \).
Then, \( q(X) = q(Y) = \dim Y \) and \( \alpha \colon Y \to A \) is an isomorphism.
Since the natural pullback homomorphism
\( \OH^{1}(X, \SO_{X}) \to \OH^{1}(Y, \SO_{Y}) \)
is an isomorphism, the action of \( G \) on \( \OH^{1}(Y, \SO_{Y}) \) is trivial.
Therefore, every element of \( G \) acts on \( A \) as a translation.
Hence, the quotient variety \( X \isom G \backslash A \) is also an abelian variety.
Conversely, if \( X \) is an abelian variety, then \( q(X) = \dim X \).
Thus, \eqref{prop:split:2} has been proved.

We shall prove the remaining assertion \eqref{prop:split:3}:
If \( q^{\circ}(X) = \dim X \), then \( Y \) is an abelian variety
by the argument above.
Thus, we may assume that \( q^{\circ}(X) < \dim X \).
Then the fiber \( F \) of \( \alpha \) is positive-dimensional and
has only canonical singularities with \( K_{F} \sim 0 \).
If \( q^{\circ}(F) > 0 \), then applying the same argument above to \( F \),
we have a finite covering \( A_{0} \times F_{0} \to F \) \'etale in codimension one
for a positive-dimensional abelian variety \( A_{0} \)
and a normal projective variety \( F_{0} \). Thus, we have a finite covering
\( A' \times A_{0} \times F_{0} \to X \) \'etale in codimension one and
a contradiction by
\[ q^{\circ}(X) = q(Y) = \dim A'
< \dim A' + \dim A_{0} \leq q(A' \times A_{0} \times F_{0}) \leq q^{\circ}(X). \]
Therefore, \( F \) is a weak Calabi--Yau variety.
Hence, the covering \( F \times A' \isom Y \times_{A} A' \to Y \to X \) satisfies
the required condition of \eqref{prop:split:3}.
Thus, we are done.
\end{proof}

\begin{corollary}\label{cor:q0CY}
Let \( S \) be a weak Calabi--Yau variety and \( A \) an abelian variety.
Then \( q^{\circ}(A \times S) = \dim A  \).
\end{corollary}

\begin{proof}
We have \( q^{\circ}(A \times S) \geq q(A \times S) = q(A) + q(S) = \dim A \).
Assume that \( q^{\circ}(A \times S) > \dim A \).
By Proposition~\ref{prop:split}, there is a finite surjective morphism
\( \tau \colon A' \times S' \to A \times S \) \'etale in codimension one
for an abelian variety \( A' \) of dimension \( q^{\circ}(A \times S) \) and
a weak Calabi--Yau variety \( S' \).
Since the first projections \( p_{1} \colon A \times S \to A \) and
\( p'_{1} \colon A' \times S' \to A' \)
are the Albanese maps, we have a surjective
morphism \( \varphi \colon A' \to A \) such that
\( p_{1} \circ \tau = \varphi \circ p'_{1}\).
Let \( B \) be a connected component of
the fiber \( \varphi^{-1}(P) \) for a general point \( P \in A \).
Then, \( B \) is a positive-dimensional abelian variety.
By restricting \( \tau \),
we have a finite covering \( B \times S' \to \{P\} \times S \)
\'etale in codimension one. This contradicts that \( q^{\circ}(S) = 0\).
Thus, \( q^{\circ}(A \times S) = \dim A \).
\end{proof}

We have the following variant of \cite{NZ},
Proposition~4.3, which treats only \'etale coverings and
varieties with only canonical singularities:

\begin{lemma}\label{lem:AlbClos}
Let \( V \) be a normal projective variety with only log-terminal singularities
such that \( K_{V} \sim_{\BQQ} 0 \).
Then there exists a finite morphism \( \tau \colon V\sptilde \to V \) satisfying
the following conditions, uniquely up
to isomorphism over \( V \)\emph{:}
\begin{enumerate}
\item \label{lem:AlbClos:1}
\( \tau \) is \'etale in codimension one.

\item \label{lem:AlbClos:2}
\( q^{\circ}(V) = q(V\sptilde) \).

\item \label{lem:AlbClos:3}
\( \tau \) is Galois.

\item \label{lem:AlbClos:4}
If \( \tau' \colon V' \to V \) satisfies the conditions
\eqref{lem:AlbClos:1}, \eqref{lem:AlbClos:2}, then
there exists a finite surjective morphism
\( \sigma \colon V' \to V\sptilde \)
\'etale in codimension one such that
\( \tau' = \tau \circ \sigma \).
% \( \tau' = \tau \circ \sigma \) for a finite surjective morphism
% \( \sigma \colon V' \to V\sptilde \) \'etale in codimension one.
\end{enumerate}
\end{lemma}

We call \( \tau \) the \emph{Albanese closure} of \( V \) in codimension one.

\begin{proof}
The same argument as in the proof of \cite{NZ}, Proposition~4.3 works as follows:
We may assume that \( q^{\circ}(V) > 0 \).
There is a Galois covering \( W \to V \) \'etale
in codimension one with \( q(W) = q^{\circ}(V) \)
by Proposition~\ref{prop:split}, \eqref{prop:split:1}.
Then \(K_{W} \sim_{\BQQ} 0 \) and \( W \) has only
log-terminal singularities (cf.\ Remark~\ref{rem:logterm}).
Let \( W \to \Alb(W) \) be the Albanese map of \( W \)
and \( \Gal(W/V) \) the Galois group of \( W \to V \).
Then we have a natural homomorphism \( \Gal(W/V)
\to \Aut(\OH_{1}(\Alb(W), \BZZ)) \). Let \( G_{0} \) be the kernel
and let \( W_{0} \) be the quotient variety \( G_{0} \backslash W \)
of \( W \) by the action of \( G_{0} \).
Then \( q(W_{0}) = q(W) \), since the quotient variety of
\( \Alb(W) \) by \( G_{0} \) is an abelian variety,
as in the proof of Proposition~\ref{prop:split}.
Therefore, the Galois covering
\( W_{0} \to V\) satisfies the
conditions~\eqref{lem:AlbClos:1}--\eqref{lem:AlbClos:3}. %%%
Let \( W' \to V \) be an arbitrary covering satisfying the
conditions~\eqref{lem:AlbClos:1} and \eqref{lem:AlbClos:2}.
Then there exist finite morphisms \( W'' \to W \) and \( W'' \to W' \)
over \( V \) such that the composite \( W'' \to V \) is Galois and
\'etale in codimension one.
Thus, \( W''_{0} \isom W_{0} \) for the quotient variety \( W''_{0} \)
for \( W'' \) obtained by the same procedure as in defining \( W_{0} \) from \( W \),
and there is a morphism \( W' \to W_{0} \) over \( V \).
Hence, \( V\sptilde := W_{0} \) satisfies all the required conditions %%%
\eqref{lem:AlbClos:1}--\eqref{lem:AlbClos:4},
and \(V\sptilde \to V \) is unique up to non-canonical isomorphism over \( V \).
\end{proof}

%%%% Q-abelian
\begin{dfn}\label{dfn:Qab}
A normal projective variety \(W\) is called \emph{Q-abelian}
if there are an abelian variety \(A\) and a finite surjective morphism
\(A \to W\) which is \'etale in codimension one.
\end{dfn}

\begin{remark}
By Proposition~\ref{prop:split},
a Q-abelian variety is characterized as a normal projective variety \( X \)
with only log-terminal singularities
such that \( K_{X} \sim_{\BQQ} 0 \) and \( q^{\circ}(X) = \dim X \).
The Albanese closure of a Q-abelian variety is abelian,
by Proposition~\ref{prop:split}, \eqref{prop:split:2}.
\end{remark}

A surjective endomorphism of the direct product of certain varieties is split.
The following gives an example:

\begin{lemma}\label{split}
Let \(A\) be an abelian variety and \(S\) a normal projective variety
with only rational singularities.
%such that \(q(S) = 0\).
Suppose that \(q(S) = 0\) and that \(S\) is not uniruled.
Let \(f \colon S \times A \to S \times A\) be a surjective morphism.
Then \(f = f_S \times f_A\) for suitable endomorphisms
\(f_S\) and \(f_A\) of \(S\) and \(A\), respectively.
\end{lemma}

\begin{proof}
By the universality of the Albanese map,
\(f\) induces a surjective endomorphism \(f_A\) of
\(A = \Alb(S \times A)\).
We can write the endomorphism \( f \) as
\( S \times A \ni (s, a) \mapsto f(s, a) = (\rho_a(s), f_{A}(a))\),
where \(\rho \colon A \to \Sur(S)\), \(a \mapsto \rho_a\),
is a morphism into
\[\Sur(S) := \{g \colon S \to S \mid  g \text{ is a surjective morphism} \}. \]
By \cite{Ho}, Theorem 3.1, the compact subvariety \(\Imm(\rho)\) is contained in the
orbit of some \(f_S \in \Sur(S)\) by the action of \(\Aut^0(S) \).
For a smooth model \( S' \) of \( S \), the birational
automorphism group \( \Bir(S') \) contains \( \Aut^{0}(S) \) as a
subgroup.  By \cite{Hm}, Theorem (2.1),
\(\Bir(S')\) is a disjoint union of abelian varieties of dimension equal
to \(q(S') = q(S) = 0\).
Thus \(\Imm(\rho)\) consists of a single element, say \(\{f_S\}\).
Then \( f = f_S \times f_A \).
\end{proof}

\section{The non-uniruled case}

In this section, we shall study non-isomorphic quasi-polarized endomorphisms of non-uniruled normal
projective varieties.
The following gives examples of such polarized endomorphisms:
% We begin with an example
% of a non-uniruled singular variety \(X\) with a non-isomorphic polarized
% endomorphism. This \(X\) has the property that \(K_X \sim_{\BQQ} 0\),
% which turns out to be a general phenomenon as seen in Theorem~\ref{thm:nonuniruled}.

\begin{example}
Let \( A \) be an abelian variety of dimension \( n \geq 2\) and let
\( H \) be a symmetric ample divisor, i.e., \( H \) is ample and
\( \iota^{*}H \sim H \) for the involution
\( \iota \colon x \mapsto -x \).
Then the multiplication map \(
\mu_{m} \colon A \ni x \mapsto mx = x + \cdots + x \in A \) by an
integer \( m \ne 0\) is polarized by \( H \) as
\( \mu_{m}^{*}H \sim m^{2}H \)
(cf.\ \cite{Mm}, Chapter~II, \S~6, Corollary~3).
Let \( X = A/\iota \) be the quotient variety by the involution
\( \iota \).
Then \( \mu_{m} \) descends to a polarized endomorphism \( f_{m} \) of \( X \)
of degree \( \deg \mu_{m} = m^{2n} \).
If \( n = 2 \), then \( X \) has only \( 16 \) rational double
points of type \( \mathsf{A}_{1} \) as singularities and its minimal resolution
of singularities is a K3 surface, called the Kummer surface of \( A \);
in particular, \( f_{m} \) for \( m > 1 \) is not nearly \'etale in the sense of
\cite{NZ}, Definition~3.2 (cf.\ \cite{NZ}, Example~3.14).
If \( n \geq 3 \), then \( X \) has only \( 2^{2n} \) terminal cyclic quotient
singular points of type \( \frac{1}{2}(1, 1, \ldots, 1) \) as singularities,
and \( 2K_{X} \sim 0 \).
Thus, \( X \) is not uniruled and \( f_{m} \) is a non-isomorphic polarized endomorphism
for \( m > 1 \).
\end{example}

In the examples above,
\( X \) has only canonical singularities and \(K_X \sim_{\BQQ} 0\).
These properties hold in general by the following
fundamental result:
%
%
% The following is a fundamental result.

\begin{theorem}\label{thm:nonuniruled}
Let \( f \colon V \to V \) be a surjective endomorphism of
a normal projective variety \( V \) and let \( H \) be
a nef and big Cartier divisor on \( V \)
such that \( f^{*}H \sim \q H \) for a positive integer
\( \q > 1 \). Suppose that \( V \) is not uniruled.
Then, there exist a projective birational morphism
\( \sigma \colon V \to X \) onto a normal projective variety \( X \),
an endomorphism \( f_{X} \) of \( X \), and
an ample divisor \( A \) on \( X \) such that
\begin{enumerate}
\item  \( X \) has only canonical singularities with
\( K_{X} \sim_{\BQQ} 0 \),

\item \( f_{X}^{*}A \sim \q A \),

\item  \( f_{X} \circ \sigma = \sigma \circ f \),

\item  \( H \sim \sigma^{*}A \), and

\item \( f_{X} \) is \'etale in codimension one.
\end{enumerate}
In particular, if \( H \) is ample, then
\( V \) has only canonical singularities, \( K_{V} \sim_{\BQQ} 0 \), and
\( f \) is \'etale in codimension one.
\end{theorem}

\begin{proof}
We may assume that \( V \) is of dimension \( n \geq 2 \).
Taking intersection numbers with \( (f^{*}(H))^{n-1} =
f^{*}(H) \cdots f^{*}(H)\)
of the both sides
of the ramification formula: \( K_{V} = f^{*}(K_{V}) + R_{f} \),
we obtain
\[ (\q-1)K_{V}H^{n-1} + R_{f}H^{n-1} = 0. \]
Thus, \( K_{V}H^{n-1} \leq 0 \).
Let \( \mu \colon Y \to V \) be a birational morphism from a smooth
projective variety \( Y \).
Since \( Y \) is not uniruled, \( K_{Y}(\mu^{*}H)^{n-1} \geq 0 \) by
\cite{MM}. Thus, \( K_{Y}(\mu^{*}H)^{n-1} = K_{V}H^{n-1} = R_{f}H^{n-1} = 0 \).
Moreover, \( K_{Y} \) is pseudo-effective
by \cite{BDPP} (cf.\ \cite{LBook2}, \S 11.4.C).
Thus, we have the \( \sigma \)-decomposition
\( K_{Y} = P_{\sigma}(K_{Y}) + N_{\sigma}(K_{Y}) \) in the sense of
\cite{Ny04}: \( N_{\sigma}(K_{Y}) \) is an effective \( \BRR \)-divisor
determined by the following property:
\( P_{\sigma}(K_{Y}) = K_{Y} - N_{\sigma}(K_{Y}) \) is movable, and
if \( B \) is an effective \( \BRR \)-divisor such that \( K_{Y} - B \) is
movable, then \( N_{\sigma}(K_{Y}) \leq B \).
Here, an \( \BRR \)-divisor \( D \) is called \emph{movable} %%%
if: for any \( \varepsilon > 0 \), any ample divisor \( H' \)
and any prime divisor \( \Gamma \),
there is an effective \( \BRR \)-divisor \( \Delta \) such that
\( \Delta \numeq D + \varepsilon H'\) and \( \Gamma \not\subset \Supp \Delta \)
(cf.\ \cite{Ny04}, Chapter III, \S 1.b). %%%
In particular, \( P_{\sigma}(K_{Y}) \) satisfies the condition
\eqref{lem:HI:1} of Lemma~\ref{lem:HI}.
Furthermore,
\[ 0 \leq P_{\sigma}(K_{Y}) (\mu^{*}H)^{n-1} =
(K_{Y} -N_{\sigma}(K_{Y})) (\mu^{*}H)^{n-1}
=  -N_{\sigma}(K_{Y}) (\mu^{*}H)^{n-1} \leq 0. \]
Therefore, \( P_{\sigma}(K_{Y}) \numeq 0 \) and
\( K_{Y} \numeq N_{\sigma}(K_{Y}) \) by Lemma~\ref{lem:HI}.
This implies that the numerical Kodaira dimension \( \kappa_{\sigma}(Y) \)
of \( Y \) in the sense of \cite{Ny04}, Chapter V, is zero.
Namely, for any ample divisor \( H' \) on \( Y \),
the function \( m \mapsto \dim \OH^{0}(Y, \SO_{Y}(mK_{Y} + H')) \) is bounded
(cf.\ \cite{Ny04}, Chapter V, Corollary~1.12).
By Theorem~4.8 of \cite{Ny04}, Chapter V,
which is the abundance theorem for \( \kappa_{\sigma} = 0 \),
we have \( \kappa(Y) = 0 \).
In particular, \( K_{Y} \sim_{\BQQ} E \) for an effective
\( \BQQ \)-divisor \( E \) such that \( E (\mu^{*}H)^{n-1} = 0 \).
Therefore, \( K_{Y} + \mu^{*}H \) has a Zariski-decomposition
whose negative part is \( E \) and whose positive part is \( \BQQ
\)-linearly equivalent to \( \mu^{*}H \)
by \cite{Ny04}, Chapter III, Proposition~3.7, i.e.,
\( N_{\sigma}(K_{Y} + \mu^{*}H) = E\), and
\( P_{\sigma}(K_{Y} + \mu^{*}H) = \mu^{*}H \) is nef.
The Zariski-decomposition above in the sense of \cite{Ny04} coincides with the
Zariski-decomposition in the sense of
Cutkosky--Kawamata--Moriwaki (\cite{Cu}, \cite{Ka87}, \cite{Mw}) or
that of Fujita \cite{Ft}, since the divisor \( K_{Y} + \mu^{*}H \)
is big (cf.\ \cite{Ny04}, Chapter III, Remark~1.17).
Thus, the positive part \( P_{\sigma}(K_{Y} + \mu^{*}H) = \mu^{*}H \)
is semi-ample, and furthermore, \( \operatorname{Bs}|m\mu^{*}H| = \emptyset  \)
for \( m \gg 0 \), by a version of
the base point free theorem (cf.\ \cite{Ft}, (A.5); \cite{Ka87}, Theorem~1;
\cite{Mw}, Theorem~0).
In particular, \( \operatorname{Bs}|mH| = \emptyset \) for \( m \gg 0 \).
Let \( \sigma \colon V \to X \) be the birational morphism onto a normal
projective variety \( X \) defined by the free linear system \( |mH|
\) for \( m \gg 0 \). Then \( H \sim \sigma^{*}A \) for
an ample divisor \( A \) on \( X \). Since \( (\mu_{*}E)H^{n-1} = R_{f}H^{n-1} = 0
\), \( \mu_{*}E \) and \( R_{f} \) are \( \sigma \)-exceptional.
Therefore, \( K_{X} = \sigma_{*}(K_{V}) \sim_{\BQQ} \sigma_{*}(\mu_{*}(E)) = 0 \) and
\( X \) has only canonical singularities, since
\( K_{Y} - \mu^{*}\sigma^{*}(K_{X}) = E \geq 0 \). %%%
By considering the Stein factorization of the
composite \( \sigma \circ f \colon V \to X\), we have an
endomorphism \( f_{X} \) of \( X \) such that
\( f_{X} \circ \sigma = \sigma \circ f \) and \( f_{X}^{*}A \sim \q A \).
Moreover, \( R_{f_{X}} = \sigma_{*}(R_{f}) = 0 \).
Hence, \( f_{X} \) is \'etale in codimension one.
The last assertion follows immediately,
since \( \sigma \) is isomorphic if \( H \) is ample.
Thus, we are done.
\end{proof}

The following gives a sufficient condition for a normal projective
variety admitting polarized endomorphisms to be Q-abelian
(cf.\ Definition~\ref{dfn:Qab}):

\begin{theorem}\label{thm:Qab}
Let \( f \colon X \to X\) be a non-isomorphic
polarized endomorphism of an \( n \)-dimensional normal
projective variety \( X \).
Assume that \( f \) is \'etale in codimension one
and that, for any point \( P \in \Sing X \), there is a connected analytic
open neighborhood \( \SU \) of \( P \)
such that the algebraic fundamental group
\( \pi_{1}^{\alg}(\SU_{\reg}) \) is finite.
Then \( X \) is a Q-abelian variety.
\end{theorem}

\begin{proof}
% Let \( A \) be an ample divisor on \( X \) such that \( f^{*}A
% \sim \q A \).
For a positive integer \( k \),
let \( \theta_{k} \colon V_{k} \to X \) be the Galois
closure of \( f^{k} \), and let \( \tau_{k} \), \(\theta_{k} \), \( g_{k} \),
and \(h_{k} \) be as in Lemma~\ref{lem:Galcl},
which are all \'etale in codimension one (cf.\ Remark~\ref{etale1}).
% We set \( A_{k} \) to be the ample divisor
% \(\tau_{k}^{*}A \).
% Then \( g_{k}^{*}A_{k} \sim A_{k+1} \) and \( h_{k}^{*}A_{k} \sim \q A_{k+1} \).
Then, \( g_{k} \) and \( h_{k} \) are both \'etale for \( k \gg 0 \) by
the claim below applied to the cases \( (\alpha_{k}, \gamma_{k}) = (\theta_{k}, h_{k}) \)
and \( (\alpha_{k}, \gamma_{k}) = (\tau_{k}, g_{k}) \).
\renewcommand{\qedsymbol}{}
\end{proof}

\begin{cclaim}
For the \( X \) above, let \( \alpha_{k} \colon V_{k} \to X \) be finite Galois coverings
and let \( \gamma_{k} \colon V_{k+1} \to V_{k} \) be
finite surjective morphisms defined for \( k \geq 1 \)
such that
\( \alpha_{k} \) and \( \gamma_{k} \) are \'etale in codimension one and
\( \alpha_{k+1} = \alpha_{k} \circ \gamma_{k} \) for \( k \geq 1 \).
Then, \( \gamma_{k} \) is \'etale
for \( k \gg 0 \).
\end{cclaim}

\begin{proof}
For a point \( P \in \Sing X \), let \( \SU \subset X\)
be a connected analytic open neighborhood
such that \( \pi_{1}^{\alg}(\SU_{\reg}) \) is finite.
For a point \( Q \in \alpha_{k}^{-1}(P) \), let \( \SV \) (depending on \( Q \) and \( k \)) be the
connected component of \( \alpha_{k}^{-1}(\SU) \) containing \( Q \).
We set
\[ \Pi(\SU; k) := \pi_{1}^{\alg}(\SV \setminus \alpha_{k}^{-1}(\Sing X)).\]
Note that \( \Pi(\SU; k) = \pi_{1}^{\alg}(\SV_{\reg}) \)
by Remark~\ref{etale1}. Since \( \alpha_{k} \) is Galois,
\( \Pi(\SU; k) \) is independent of the choice of
\( Q \in \alpha_{k}^{-1}(P)\) and is a normal subgroup of
\( \pi_{1}^{\alg}(\SU_{\reg}) \).
By the finiteness assumption of \( \pi_{1}^{\alg}(\SU_{\reg})  \),
we have a positive integer \( k_{P} \)
such that the injection
\( \gamma_{k*} \colon \Pi(\SU; k+1) \to \Pi(\SU; k)  \)
is isomorphic for any \( k \geq k_{P} \).
As a consequence, we infer that
\( \gamma_{k} \colon V_{k+1} \to V_{k} \) is \'etale
along \( \alpha_{k+1}^{-1}(P) \) for any \( k \geq k_{P} \).
Since \( \Sing X \) is compact, we can find a positive integer \( k_{0} \) such that
\( \gamma_{k} \colon V_{k+1} \to V_{k} \)
is \'etale for any \( k \geq k_{0} \).
\end{proof}

\begin{proof}[Proof of Theorem~\ref{thm:Qab} continued]
% We can apply the claim to the cases \( (\alpha_{k}, \gamma_{k}) = (\theta_{k}, h_{k}) \)
% and \( (\alpha_{k}, \gamma_{k}) = (\tau_{k}, g_{k}) \).
% Hence, \( g_{k} \) and \( h_{k} \) are both \'etale for \( k \gg 0 \).
% We fix such a large positive integer \( k \).
We fix a large positive integer \( k \)
such that \( g_{k} \) and \( h_{k} \) are both \'etale.
We shall show that \( V_{k} \) is smooth. Assume the contrary that
\( \Sing V_{k} \ne \emptyset \).
We set \( d := \dim \Sing V_{k} \). Then \( 0 \leq d \leq n - 2 \).
Since
\( g_{k}^{-1}(\Sing V_{k}) = h_{k}^{-1}(\Sing V_{k}) = \Sing V_{k+1} \),
the mapping degrees of
\( g_{k} \colon \Sing V_{k+1} \to \Sing V_{k} \) and
\( h_{k} \colon \Sing V_{k+1} \to \Sing V_{k} \) are
\( \deg g_{k} \) and \( \deg h_{k} \), respectively.
% We set \( d := \dim \Sing V_{k} \leq n - 2 \).
Then \( d > 0 \); otherwise, we have a contradiction by
\( \sharp \Sing V_{k+1}
= (\deg g_{k})\sharp \Sing V_{k} = (\deg h_{k}) \sharp \Sing V_{k} \)
and \( \deg h_{k} > \deg g_{k} \).
% Thus \( d \geq 1 \).
Let \( A \) be an ample divisor on \( X \) such that \( f^{*}A
\sim \q A \) for an integer \( \q > 1 \).
We set \( A_{l} \) to be the ample divisor
\(\tau_{l}^{*}A \) for any \( l \geq 1 \).
Then \( g_{k}^{*}A_{k} \sim A_{k+1} \) and \( h_{k}^{*}A_{k} \sim \q A_{k+1} \).
Hence, we have the equalities
\begin{align*}
(\Sing V_{k+1})A_{k+1}^{d} &=
(\Sing V_{k+1}) g_{k}^{*}(A_{k})^{d} = (\deg g_{k})(\Sing V_{k})A_{k}^{d} \\
&= \q^{-d}(\Sing V_{k+1})h_{k}^{*}(A_{k})^{d} = \q^{-d}(\deg h_{k})
(\Sing V_{k})A_{k}^{d}
\end{align*}
of intersection numbers.
% since \( g_{k}^{*}A_{k} \sim A_{k+1} \)
% and \( h_{k}^{*}A_{k} \sim \q A_{k+1} \).
Thus, \( (\Sing V_{k})A_k^{d} = 0 \) by
\( \deg h_{k} = (\deg f)(\deg g_{k}) = \q^{n} \deg g_{k}\) and \( d < n \);
this is absurd, since \( A_k \) is ample.
Consequently, \( V_{k} \) is smooth.
Since \( g_{k} \), \(h_{k} \colon V_{k+1} \to V_{k} \) are \'etale
morphisms with \( \deg h_{k} > \deg g_{k} \), we have
\[ c_{1}(V_{k})A_{k}^{n-1} = c_{1}(V_{k})^{2}A_{k}^{n - 2} =
c_{2}(V_{k})A_{k}^{n-2} = 0\]
by a similar calculation of intersection numbers as above.
Then \( c_{1}(V_{k}) \) is numerically trivial
by the hard Lefschetz theorem.
Moreover, the vanishing of
\( c_{2}(V_{k})A_{k}^{n-2} \)
implies that an \'etale covering of \( V_{k} \)
is an abelian variety by \cite{Yau} (cf.\ \cite{Be1}).
Therefore, \( X \) is a Q-abelian variety.
\end{proof}

\begin{remark}
An argument on the Galois closure
in the proof of Theorem~\ref{thm:Qab} is borrowed from
\cite{ENS}.
A result of Campana \cite{Cp04}, Corollary~6.3, gives
another proof of Theorem~\ref{thm:Qab} in the case where
\( K_{X} \sim_{\BQQ} 0 \) and \( X \) has only quotient
singularities.
\end{remark}

Applying Theorems~\ref{thm:nonuniruled} and \ref{thm:Qab},
we have the following partial answer to Conjecture~\ref{conj:Qab}.

\begin{theorem}\label{thm:solQab}
Let \( X \) be a non-uniruled normal projective variety such that
\( \dim X \leq 3 \) or that \( X \) has only quotient singularities.
If \( X \) admits a non-isomorphic polarized endomorphism,
then \( X \) is Q-abelian.
\end{theorem}

\begin{proof}
By Theorem~\ref{thm:Qab}, it is enough to show that
any singular point has a connected analytic
open neighborhood \( \SU \) such that
\( \pi_{1}^{\alg}(\SU_{\reg}) \) is finite.
If \( X \) has only quotient singularities, then this is true.
We know that \( X \) has only canonical singularities
by Theorem~\ref{thm:nonuniruled}.
If \( \dim X \leq 2\), then \( X \) has only quotient singularities.
If \( \dim X = 3 \), then the finiteness of \( \pi_{1}^{\alg}(\SU_{\reg}) \)
is proved in \cite{SW}, Theorem~3.6.
Thus, we are done.
\end{proof}

Even though Conjecture~\ref{conj:Qab} is still open,
we have the following:

\begin{proposition}\label{prop:qXfqX}
Let \( X \) be a normal projective variety and \( f \colon X \to X \)
a polarized endomorphism of \( X \)
such that \( \deg f = \q^{\dim X} \) for an integer \( \q \ge 1 \).
Assume that \( X \) has only log-terminal singularities and
\( K_{X} \sim_{\BQQ} 0 \). Then there exist a finite covering
\( \tau \colon A \times S \to X \) \'etale in codimension one for an abelian variety
\( A \) and a weak Calabi--Yau variety \( S \), and polarized endomorphisms
\( f_{A} \colon A \to A \), \( f_{S} \colon S \to S \) such that
\( \deg f_{A} = \q^{\dim A}\), \(\deg f_{S} = \q^{\dim S} \), and
\( \tau \circ (f_{A} \times f_{S}) = f \circ \tau \)\emph{:}
\[ \begin{CD}
A \times S @>{f_{A} \times f_{S}}>> A \times S \\
@V{\tau}VV @V{\tau}VV\\
X @>{f}>> \phantom{.}X.
\end{CD} \]
In particular,
\( q^{\circ}(X) = q^{\natural}(X) = q^{\circ}(X, f) = q^{\natural}(X, f)
= q(A) = \dim A\).
\end{proposition}

\begin{proof}
By Lemma~\ref{lem:first} and its corollary,
there is an ample divisor \( H \) on \( X \) such that \( f^{*}(H) \sim \q H \).

\emph{Step}~1. \emph{Reduction to the case
where \( X \) has only canonical singularities with \( K_{X} \sim 0 \)}:
Let \( \nu \colon \widehat{X} \to X \) be the global index-one cover, i.e.,
the minimal cyclic covering satisfying \( K_{\widehat{X}} \sim 0 \)
(cf.\ the proof of Proposition~\ref{prop:split}).
Then \( \widehat{X} \) has only canonical singularities
by \cite{Ka}, Proposition~1.7.
By the uniqueness of the global index-one cover, there is an
endomorphism \( \hat{f} \colon \widehat{X} \to \widehat{X}\) such that
\( \nu \circ \hat{f} = f \circ \nu \). This is shown as follows:
For the normalization \( X^{\flat} \) of
the fiber product \( \widehat{X} \times_{X} X \) of
\( \nu \) and \( f \) over \( X \), let
\( p^{\flat}_{1} \colon X^{\flat} \to \widehat{X} \) and
\( p^{\flat}_{2} \colon X^{\flat} \to X \) be the morphisms induced from the first and
second projections, respectively.
Then, the restriction \( X^{\flat}_{i} \to \widehat{X} \) of \( p_{1}^{\flat} \) to
any connected component \( X^{\flat}_{i} \) of \( X^{\flat} \) is
finite and surjective, since so is \( f \).
Thus, pulling back a nowhere vanishing section of
\( \SO_{\widehat{X}}(K_{\widehat{X}}) \),
we have a holomorphic section of \( \SO_{X^{\flat}}(K_{X^{\flat}}) \),
which is not zero on each connected component of \( X^{\flat} \).
On the other hand, \( p^{\flat}_{2} \)
is \'etale in codimension one, since so is \( \nu \).
Hence, \( K_{X^{\flat}} \sim 0 \).
Noting that \( \deg p^{\flat}_{2} = \deg \nu \) and
by the minimality and the uniqueness of the
global-index one covering, we infer that \( X^{\flat} \) is irreducible and that
\( p^{\flat}_{2} \colon X^{\flat} \to X \) is isomorphic to
the global index-one covering \( \nu \colon \widehat{X} \to X \) over \( X \).
Thus, \( p_{1}^{\flat} \) produces an endomorphism
\( \hat{f} \colon \widehat{X} \to \widehat{X} \)
satisfying \( \nu \circ \hat{f} = f \circ \nu \).
Then, \( \hat{f} \) is a polarized endomorphism with
\( \hat{f}^{*}(\nu^{*}(H)) \sim \q \nu^{*}(H) \), since \( f^{*}(H) \sim \q H \).
Therefore, we may assume that \( X \) has only canonical singularities
with \( K_{X} \sim 0 \) by replacing \( (X, f) \) with \( (\widehat{X}, \hat{f}) \).
Note that the replacement does not affect
the last equalities of Proposition~\ref{prop:qXfqX}
by Lemma~\ref{lem:circsharp1}.

\emph{Step}~2. \emph{Reduction to the case where \( q^{\circ}(X) = q(X) \)}:
Let \( \lambda \colon X\sptilde \to X \)
be the Albanese closure of \( X \) in codimension one
defined in Lemma~\ref{lem:AlbClos}.
By the uniqueness of \( \lambda \), \( X\sptilde \)
admits an endomorphism \( \tilde{f} \)
such that \( \lambda \circ \tilde{f} = f \circ \lambda \).
This is shown as follows:
For  the normalization \( X^{\sharp} \) of
the fiber product \( X\sptilde \times_{X} X \) of \( \lambda \)
and \( f \), let \( p^{\sharp}_{1} \colon X^{\sharp} \to X\sptilde \) and
\( p^{\sharp}_{2} \colon X^{\sharp} \to X \) be the morphisms
induced from the first and second projections, respectively.
Then the restriction \( X^{\sharp}_{i} \to X\sptilde \)
of the morphism \( p^{\sharp}_{1} \) to
any connected component \( X^{\sharp}_{i} \) of \( X^{\sharp} \)
is a finite surjective morphism, since so is \( f \).
Thus, \( q(X^{\sharp}_{i}) \geq q(X\sptilde) = q^{\circ}(X) \).
On the other hand, \( p^{\sharp}_{2} \) is \'etale in codimension one,
since so is \( X\sptilde \to X \).
By Lemma~\ref{lem:AlbClos},
the restriction \( X^{\sharp}_{i} \to X \) of \( p_{2}^{\sharp} \) to
\( X^{\sharp}_{i} \) factors through \( \lambda \colon X\sptilde \to X \).
Thus, \( \deg (X^{\sharp}_{i} / X) \geq \deg(X\sptilde/X) \).
Since \( \deg(X\sptilde/X) = \deg(X^{\sharp}/X)
= \sum_{i} \deg(X^{\sharp}_{i}/X) \), we infer that \( X^{\sharp} \) is irreducible and
\( p^{\sharp}_{2} \colon X^{\sharp} \to X\)
is isomorphic to the Albanese closure \( \lambda \colon X\sptilde \to X \)
in codimension one over \( X \).
Thus, \( p^{\sharp}_{1} \colon X^{\sharp} \to X\sptilde \) produces
an endomorphism \( \tilde{f} \colon X\sptilde \to X\sptilde\) satisfying
\( \lambda \circ \tilde{f} = f \circ \lambda \).
Then, \( \tilde{f} \) is a polarized endomorphism with
\( \tilde{f}^{*}(\lambda^{*}(H)) \sim \q\lambda^{*}(H) \), since
\( f^{*}(H) \sim \q H \).
Therefore, we may assume that \( q^{\circ}(X) = q(X) \)
by replacing \( (X, f) \) with \( (X\sptilde, \tilde{f}) \).
Note again that the replacement does not affect
the last equalities of Proposition~\ref{prop:qXfqX}
by Lemma~\ref{lem:circsharp1}.

\emph{Step}~3. \emph{The final step}:
We may assume that \( X \) has only canonical singularities, \( K_{X} \sim 0 \),
and \( q^{\circ}(X) = q(X) \), by the previous steps.
If \( q(X) = 0 \), then \( X \) is weak Calabi--Yau, and
Proposition~\ref{prop:qXfqX} holds in this case.
Thus, we may assume that \( q(X) > 0 \).
Let \( \alpha \colon X \to A := \Alb(X) \) be the Albanese map.
Then, there is an endomorphism \( f'_{A} \colon A \to A\)
such that \( \alpha \circ f = f'_{A} \circ \alpha \) by the universality of
the Albanese map. By \cite{Ka85}, Theorem~8.3, we can find an \'etale covering
\(\theta \colon T \to A\)
such that \( X \times_{A} T \isom S \times T\)
over \(T\) for a fiber \(S\) of \( \alpha \).
Here, \(S\) is weak Calabi--Yau by the definition of \(q^{\circ}(X)\).
Taking a further \'etale covering,
we may assume that \(T \isom A\) and \(\theta \colon T \to A\)
is just the multiplication map by a positive integer \( m \) for
a certain group structure of \( A \).
There is an endomorphism \( f_{A} \) of
\( A \) such that \( \theta \circ f_{A} = f'_{A} \circ \theta\)
by \cite{NZ}, Lemma~4.9.
Let \( W \) be the fiber product \( X \times_{A} A \)
of \( \alpha \colon X \to A \) and \( \theta \colon A \to A \) over \( A \), %%
and let \( \varphi \colon W \to X \) be
the finite \'etale covering induced from the first projection.
Then \( W \isom S \times A \) over \( A \) as above,
and \( f \times f_{A} \colon X \times A \to X \times A\)
induces an endomorphism \( f_{W} \) of \( W \subset X \times A\)
such that \( \varphi \circ f_{W} = f \circ \varphi\).
In particular,
\( f_{W} \) is a polarized endomorphism with
\( f_{W}^{*}(\varphi^{*}(H)) \sim \q\varphi^{*}(H) \) for the ample divisor
\( \varphi^{*}(H) \).
We have an endomorphism \(f_S \colon S \to S\) such that
\(f_W = f_S \times f_A\) by Lemma~\ref{split}.
Then, \( f_{S} \) and \( f_{A} \) are polarized endomorphisms
with \( \deg f_{A} = \q^{\dim A}\) and \(\deg f_{S} = \q^{\dim S} \)
by \cite{IntS}, Proposition~4.17.
% and Theorem~4.4.
It remains to show the last equalities.
We have \( q^{\natural}(X) = q^{\circ}(X) \geq
q^{\natural}(X, f) = q^{\circ}(X, f) \) by
Lemma~\ref{lem:circsharp1}, and
\( q^{\circ}(A \times S) = q(A) = \dim A\)
by Corollary~\ref{cor:q0CY}.
In view of the covering
\( A \times S \to X \) \'etale in codimension one,
we have
\[ q(A) \leq q^{\circ}(A \times S, f_{A} \times f_{S}) = q^{\circ}(X, f) \leq
q^{\circ}(X) = q^{\circ}(A \times S) = q(A). \]
Thus, the expected equalities also hold.
\end{proof}

Theorem~\ref{Th0} for non-uniruled \( X \) is a consequence of
Theorems~\ref{thm:nonuniruled} and \ref{thm:solQab}, and
Proposition~\ref{prop:qXfqX}.

\section{The proof of Theorems~\ref{Th0} and \ref{Th1.2new}}

The following result gives a descent property of polarized
endomorphisms by maximal rationally connected fibrations,
which is proved in \cite{IntS}, Section~4.3.

\begin{lemma}[{\cite{IntS}, Corollary~4.20}]\label{lem:MRC}
Let \( f \colon X \to X \) be a quasi-polarized endomorphism of a
normal projective variety \( X \). Suppose that \( X \) is uniruled.
Then there exist a birational morphism \( \sigma \colon W \to X \),
an equi-dimensional surjective morphism \( p \colon W \to Y \),
and quasi-polarized endomorphisms \( f_{W} \colon W \to W \), \( f_{Y}
\colon Y \to Y \)
satisfying the following conditions\emph{:}
\begin{enumerate}
\item  \( W \) and \( Y \) are normal projective varieties,
and \( \dim Y < \dim X \).

\item  \( Y \) is not uniruled if \( \dim Y > 0 \).

\item  A general fiber of \( p \) is rationally connected.

\item \( \deg f_{Y} = (\deg f)^{\dim Y/\dim X} \).

\item \( \sigma \circ f_{W} = f \circ \sigma\) and
\( p \circ f_{W} = f_{Y} \circ p \), i.e., the diagram below is commutative\emph{:}
\[ \begin{CD}
Y @<{p}<< W @>{\sigma}>> X \\
@V{f_{Y}}VV @V{f_{W}}VV @V{f}VV \\
Y @<{p}<< W @>{\sigma}>> \phantom{.}X.
\end{CD}\]

\item If \( f \) is polarized, then both \( f_W \) and \( f_{Y} \) are polarized.
\end{enumerate}

\end{lemma}

An outline of the proof of Lemma~\ref{lem:MRC} is as follows:
We take a dominant rational map \(X \ratmap Y\)
which is birational to the maximal rationally connected fibration
\( \widetilde{X} \ratmap \widetilde{Y} \)
(cf.\ \cite{Cp92}, \cite{KoMM}, \cite{GHS})
of a smooth model \( \widetilde{X} \) of \( X \).
It is determined uniquely up to birational equivalence
by the property that \( Y \) is not uniruled
(when \( \dim Y > 0 \)) %%%
and a general `fiber' of \( X \ratmap Y \) is rationally connected.
Among the choices of the rational maps
\( X \ratmap Y \), we can select a unique one up to isomorphism
by the following two properties:
\begin{itemize}
\item  The graph \(\Gamma_{X/Y}\) of \(X \ratmap Y\) is equi-dimensional over \(Y\).

\item  If \(\nu \colon Y' \ratmap Y\) is a birational map
from another normal projective variety \(Y'\)
such that the graph \(\Gamma_{X/Y'}\) of the composite
\(X \ratmap Y \overset{\nu^{-1}}{\ratmap} Y'\) of rational maps
is equi-dimensional over \(Y'\),
then \(\nu\) is holomorphic.
\end{itemize}
The existence and the uniqueness of \( X \ratmap Y \) is proved in
\cite{IntS}, Proposition~4.14 (cf.\ \cite{IntS}, Theorem~4.18).
The proof uses the notion of \emph{intersection sheaves},
which we do not explain here.
The variety \( W \) is just the normalization of \( \Gamma_{X/Y} \).
The endomorphism \( f \) descends to an endomorphism \( f_{Y} \) of \( Y \) by
\cite{IntS}, Theorem~4.19.
If \(f\) is polarized (resp. quasi-polarized) then so is \( f_{Y} \) by
\cite{IntS}, Corollary~4.20; more precisely,
if \( f^{*}(H) \sim \q H \) for an ample (resp.\ a nef and big)
divisor \( H \) on \( X \),
where \( \q = (\deg f)^{1/\dim X}\), %%%
then \( f_{Y}^{*}(H_{Y}) \sim \q H_{Y} \) for
an ample (resp.\ a nef and big) divisor \( H_{Y} \) on \( Y \).
The proofs of two assertions also use the notion of
intersection sheaves.
The endomorphism \( f_{W} \) of \( W \) is induced from
\(f \times f_{Y} \).
Since \( f^{*}(H) \sim \q H \), we have %%%
\( f_{W}^{*}(H_{W}) \sim \q H_{W} \) for the ample (resp.\ nef and big) divisor
\( H_{W} = \sigma^{*}(H) + p^{*}(H_{Y})\) for
the induced morphisms \( \sigma \colon W \to X \) and \( p \colon W \to Y \).
Thus, \( f_{W} \) is also polarized (resp.\ quasi-polarized).

\begin{rremark}\label{rem:DelignePair}
The same assertion as in Lemma~\ref{lem:MRC} %%%
for polarized endomorphisms is stated in \cite{Zs},
Proposition~2.2.4.
However, the argument there is valid
only when the maximal rationally connected fibration
is flat,  which is not a priori  available.
The study of intersection sheaves in \cite{IntS}
renders the flatness requirement redundant,
and consequently the expected assertion
is proved in \cite{IntS}, Section~4.3.
\end{rremark}

\begin{rremark}
In the situation of Lemma~\ref{lem:MRC},
assume that \( \deg f > 1\). Then, %added
there exist a birational morphism \( Y \to Y' \) onto
a normal projective variety \( Y' \) with only canonical
singularities such that \( K_{Y'} \sim_{\BQQ} 0 \) and
a polarized endomorphism
\( f_{Y'} \colon Y' \to Y' \) compatible with \( f_{Y} \) by
Theorem~\ref{thm:nonuniruled}.
Applying \cite{Fa}, Theorem~5.1 to \( f_{Y'} \),
we infer that the set \( \SY \)
of periodic points of \( f_{Y} \) is Zariski dense in \( Y \).
Here, \( y \in \SY \) if and only if
\( f_{Y}^{r}(y) = y \) for a positive integer \( r = r(y)\).
Thus, if \( y \in \SY \) is general, then
a multiple of \( f \) induces a non-isomorphic
quasi-polarized endomorphism of
the rationally connected 
%normal 
variety \( p^{-1}(y) \).
Hence the study of quasi-polarized endomorphisms on uniruled varieties
is reduced, to some extent, to that on rationally connected varieties.
%
% Thus, \( f_W^{r} \) for the same \( r \) induces a quasi-polarized endomorphism of
% the fiber \( p^{-1}(y) \) over \( y \in \SY\).
% Note that \( p^{-1}(y) \) for general \( y \in \SY \)
% is a rationally connected normal variety.
\end{rremark}

\begin{lemma}\label{lem:qYqXPart}
In the situation of \emph{Lemma~\ref{lem:MRC}},
assume that \( f  \) is
%non-isomorphic and
polarized.
Let \( \theta \colon Y' \to Y \) be a finite covering
\'etale in codimension one from a normal
variety \( Y' \) and let \( f_{Y'} \colon Y' \to Y' \) be
a polarized endomorphism such that \( \theta \circ f_{Y'} = f_{Y} \circ \theta \).
Then there exist normal projective varieties \( X' \), \( W' \),
finite coverings \( \tau \colon X' \to X \) and \( \delta \colon W' \to W \)
both \'etale in codimension one,
a birational morphism \( \sigma' \colon W' \to X' \),
a fibration \( p' \colon W' \to Y' \) whose general fiber is rationally connected,
and polarized endomorphisms \( f' \colon X' \to X' \), \( f_{W'} \colon W' \to W' \)
such that
\( \tau \circ f' = f \circ \tau\),
\( \sigma' \circ f_{W'} = f' \circ \sigma'\),
\( \delta \circ f_{W'} = f_{W} \circ \delta \), and
\( p' \circ f_{W'} = f_{Y'} \circ p' \)\emph{;} hence,
the diagram below is commutative and all the varieties
admit mutually compatible polarized endomorphisms\emph{:}
\[ \begin{CD}
Y' @<{p'}<< W' @>{\sigma'}>> X' \\
@V{\theta}VV @V{\delta}VV @V{\tau}VV \\
Y @<{p}<< W @>{\sigma}>> \phantom{.}X.
\end{CD} \]
In particular, \( \sigma' \), \( p' \), \( f_{W'} \), and \( f_{Y'} \)
satisfy the same conditions as in
\emph{Lemma~\ref{lem:MRC}} for \( f' \colon X' \to X' \).
\end{lemma}

\begin{proof}
Let \( W' \) be the normalization  of \( W \times_{Y} Y' \).
Let \( \delta \colon W' \to W \) and
\( p' \colon W' \to Y' \) be the morphisms induced from the first and
second projections, respectively.
Then, a general fiber of \( p' \) is also a rationally connected variety.
In particular, \( W' \) is connected; thus \( W' \) is a normal projective variety.
Since \( p \) is equi-dimensional, \( p' \) is also equi-dimensional and
the finite morphism \( \delta \colon W' \to W \) is \'etale in codimension one.
As the Stein factorization of the composite
\(  \sigma \circ \delta \colon W' \to W \to X \),
we have a birational morphism \( \sigma' \colon W' \to X' \) and
a finite morphism \( \tau \colon X' \to X \) for a normal projective variety \( X' \)
such that \( \tau \circ \sigma' = \sigma \circ \delta \).
Let \( U \subset X \) be the domain of \( \sigma^{-1} \colon X \ratmap W \).
% Let \( U \subset X \) be the maximum open subset such that
% \( \sigma^{-1}(U) \) is isomorphic to \( U \) via \( \sigma \colon W \to X\).
Then, \( \Codim(X \setminus U) \geq 2 \), and
the restriction \( \tau^{-1}(U) \to U\) of \( \tau \)
is \'etale in codimension one, since so is \( \delta \).
Therefore, \( \tau \colon X' \to X \) is \'etale in codimension one.

A polarized
endomorphism \( f_{W'} \colon W' \to W'\) is induced from
\( f_{W} \times f_{Y'} \) of \( W \times Y' \).
It satisfies \( \delta \circ f_{W'} = f_{W} \circ \delta \)
and \( p' \circ f_{W'} = f_{Y'} \circ p'\).
Moreover, we have relations
\[ (\tau \circ \sigma') \circ f_{W'} = (\sigma \circ \delta) \circ f_{W'} =
\sigma \circ f_{W} \circ \delta = f \circ (\sigma \circ \delta)
= f \circ (\tau \circ \sigma'). \]
Thus, the Stein factorization of
\( (\tau \circ \sigma') \circ f_{W'} \colon W' \to X \) is given by
the birational morphism \( \sigma' \colon W' \to X' \) and
the finite morphism \( f \circ \tau \colon X' \to X \).
Since \( \tau \) is finite,
the Stein factorization of \( \sigma' \circ f_{W'} \colon W' \to X' \)
is also given by the same birational morphism \( \sigma' \colon W' \to X' \).
Therefore, we have an endomorphism \( f' \colon X' \to X' \) such that
\( \sigma' \circ f_{W'} = f' \circ \sigma'\).
We have also \( \tau \circ f' = f \circ \tau \) by
the surjectivity of \( \sigma' \) and by the relation
\( (\tau \circ \sigma') \circ f_{W'} =f \circ (\tau \circ \sigma')\) above.
The endomorphism \( f' \) is polarized by
the pullback of an ample divisor on \( X \) polarizing \( f \).
Thus, we are done.
\end{proof}

\begin{lemma}\label{lem:qYqX}
In the situation of \emph{Lemma~\ref{lem:MRC}},
assume that \( f  \) is non-isomorphic and
polarized. Then\emph{:}
\begin{enumerate}
\item \label{lem:qYqX:1}
\( Y \) has only canonical singularities with \( K_{Y} \sim_{\BQQ} 0 \).

\item \label{lem:qYqX:2}
\( q(X) \leq \tilde{q}(X) = q(Y) \leq q^{\circ}(Y) = q^{\circ}(Y, f_{Y})
= q^{\natural}(Y, f_{Y}) \leq \dim Y\) and
\( q^{\natural}(Y, f_{Y}) \leq q^{\natural}(X, f) \).

\item \label{lem:qYqX:3}
The homomorphism \( p_{*} \colon \pi_{1}(W) \to \pi_{1}(Y) \)
of the fundamental groups is isomorphic
and \( \sigma_{*} \colon \pi_{1}(W) \to \pi_{1}(X) \) is surjective.
\end{enumerate}
\end{lemma}

\begin{proof}
The assertion \eqref{lem:qYqX:1} follows from Theorem~\ref{thm:nonuniruled} and
Lemma~\ref{lem:MRC}.
In particular, \( q(Y) \leq q^{\circ}(Y) = q^{\circ}(Y, f_{Y})
= q^{\natural}(Y, f_{Y}) \leq \dim Y \) by Proposition~\ref{prop:qXfqX}.
We can take birational morphisms \( \mu_{W} \colon \widetilde{W} \to W \) and
\( \mu_{Y} \colon \widetilde{Y} \to Y \)
from smooth projective varieties \( \widetilde{W} \) and \( \widetilde{Y} \),
respectively,
such that the induced rational map
\( \tilde{p} \colon \widetilde{W} \to \widetilde{Y} \) is
holomorphic and smooth over the complement of a normal crossing
divisor on \( \widetilde{Y} \).
Then, \( q(\widetilde{W}) \geq q(W) \geq q(X) \) by the injections
\( \sigma^{*} \colon \OH^{1}(X, \SO_{X}) \to \OH^{1}(W, \SO_{W}) \) and
\( \mu_{W}^{*} \colon \OH^{1}(W, \SO_{W}) \to
\OH^{1}(\widetilde{W}, \SO_{\widetilde{W}}) \).
On the other hand, \( q(\widetilde{Y}) = \tilde{q}(Y) = q(Y) \), since \( Y \) has
only rational singularities.
% For the fibration \( \tilde{p} \),
We have
\( \OR^{1}\tilde{p}_{*}\SO_{\widetilde{W}} = 0 \) by Koll\'ar's
torsion free theorem \cite{Ko86}, since
a general fiber of \( \tilde{p} \) is rationally connected.
Hence, \( q(\widetilde{W}) = q(\widetilde{Y}) \), consequently,
\( q(Y) = \tilde{q}(X) \geq q(X) \).
For the proof of \eqref{lem:qYqX:2},
it remains to show the inequality: \( q^{\circ}(Y, f_{Y}) \leq q^{\natural}(X, f) \).
We can take a normal projective variety \( Y' \),
a finite covering \( \theta \colon Y' \to Y \) \'etale in codimension one,
and an endomorphism \( f_{Y'} \colon Y' \to Y' \)
such that \( \theta \circ f_{Y'} = f_{Y} \circ \theta\) and
\( q(Y') = q^{\circ}(Y, f_{Y}) \). By Lemma~\ref{lem:qYqXPart}, we have
a normal projective variety \( X' \), a finite covering \( \tau \colon X' \to X \)
\'etale in codimension one, and
an endomorphism \( f' \colon X' \to X' \) such that
\( \tau \circ f' = f \circ \tau \) and \( f_{Y'} \) is
obtained from \( f' \) as in Lemma~\ref{lem:MRC}.
Hence, \( q^{\natural}(X, f) \geq \tilde{q}(X') = q(Y') \) by the argument above.
Thus, the assertion \eqref{lem:qYqX:2} has been proved.

Next, we shall prove \eqref{lem:qYqX:3}. Let \( U \subset X \) be a Zariski open dense
subset of \( X \) such that
\( \mu_{W}^{-1}\sigma^{-1}(U) \isom \sigma^{-1}(U) \isom  U\) for the birational morphisms
\( \sigma \) and \( \mu_{W} \).
Note that the homomorphism \( \pi_{1}(U) \to \pi_{1}(X) \) associated with
the open immersion \( U \injmap X \) is surjective, since \( X \) is normal.
Similarly, \( \pi_{1}(\sigma^{-1}(U)) \to \pi_{1}(W) \) and
\( \pi_{1}(\mu_{W}^{-1}\sigma^{-1}(U)) \to \pi_{1}(\widetilde{W}) \)
are surjective. Thus,
\( \mu_{W *} \colon \pi_{1}(\widetilde{W}) \to \pi_{1}(W) \) and
\( \sigma_{*} \colon \pi_{1}(W) \to \pi_{1}(X) \)
are surjective. On the other hand, by \cite{Ty},
\( \mu_{Y*} \colon \pi_{1}(\widetilde{Y}) \to \pi_{1}(Y) \) is an isomorphism,
since \( Y \) has only canonical singularities.
Moreover, \( \tilde{p}_{*} \colon \pi_{1}(\widetilde{W}) \to \pi_{1}(\widetilde{Y}) \)
is an isomorphism, by \cite{Ko93}, Theorem~5.2 (cf.\ \cite{NZ}, Lemma~5.3).
Hence, \( \pi_{1}(\widetilde{W}) \isom \pi_{1}(W) \isom
\pi_{1}(\widetilde{Y}) \isom \pi_{1}(Y) \). Thus the assertion
\eqref{lem:qYqX:3} has been proved.
\end{proof}

\begin{corollary}\label{cor:qXqY}
Let \( X \) be an \( n \)-dimensional normal projective
variety admitting a non-isomorphic
polarized endomorphism \( f \colon X \to X\).
Then\emph{:}
\begin{enumerate}
\item \label{cor:qXqY:1}
The inequality \( q^{\natural}(X, f) \leq n \) holds, in which
the equality holds if and only if \( X \) is Q-abelian.

\item \label{cor:qXqY:2} %%%changed
If \emph{Conjecture~\ref{conj:Qab}} is true for
the varieties of dimension at most \( n - q^{\natural}(X, f) \),
% \( n\),
then \( \pi_{1}(X) \) contains a finite-index subgroup
which is a finitely generated abelian group of rank
at most \( 2q^{\natural}(X, f) \).
%
% \item \label{cor:qXqY:4}
% If \( q^{\natural}(X, f) \geq n - 3 \), then the conclusion of \eqref{cor:qXqY:3}
% holds.
\end{enumerate}
\end{corollary}

\begin{proof}
\eqref{cor:qXqY:1}:
Suppose that \( X \) is not uniruled.
Then, \( X \) has only canonical singularities and \( K_{X} \sim_{\BQQ} 0 \) by
Theorem~\ref{thm:nonuniruled}.
Hence, \( q^{\natural}(X, f) \leq n \) in which the equality holds if and only if
\( X \) is Q-abelian, by Proposition~\ref{prop:qXfqX}.
Therefore, we have only to show \( q^{\natural}(X, f) < n \)
assuming that \( X \) is uniruled.
By replacing \( X \) with a finite covering \( \widehat{X} \to X \)
\'etale in codimension one and by replacing \( f \) with
an endomorphism of \( \widehat{X} \) compatible with the original \( f \),
we may assume that \( q^{\natural}(X, f) = \tilde{q}(X) \).
Then, for the morphisms \( \sigma \colon W \to X \) and \( p \colon W \to Y \)
in Lemma~\ref{lem:MRC},
we have \( \tilde{q}(X) = q(Y) \leq \dim Y < \dim X = n\) by
Lemma~\ref{lem:qYqX}, \eqref{lem:qYqX:2}.
Thus, the assertion \eqref{cor:qXqY:1} has been proved.
%
% Let \( \hat{\tau} \colon \widehat{X} \to X \) be
% a finite covering \'etale in codimension one from a
% normal projective variety \( \widehat{X} \) and
% \( \hat{f} \colon \widehat{X} \to \widehat{X} \) an endomorphism such that
% \( \hat{\tau} \circ \hat{f} = f \circ \hat{\tau} \).
% Then, \( \tilde{q}(\widehat{X}) \leq n \) by
% Lemma~\ref{lem:qYqX}, \eqref{lem:qYqX:2}, applied to \( \widehat{X} \) and
% \( \hat{f} \).
% Hence, \( q^{\natural}(X, f) \leq n \).
% If \( X \) is Q-abelian, then \( q^{\natural}(X, f) = n \) by
% Proposition~\ref{prop:qXfqX}.
% We shall show the converse
% that if \( q^{\natural}(X, f) = n \), then \( X \) is Q-abelian.
% Replacing \( X \) with a finite covering \( \widehat{X} \to X \)
% \'etale in codimension one and \( f \) with
% an endomorphism of \( \widehat{X} \) compatible with \( f \),
% we may assume that \( q^{\natural}(X, f) = \tilde{q}(X) = n \).
% Suppose that \( X \) is uniruled.
% Then, for the morphisms \( \sigma \colon W \to X \) and \( p \colon W \to Y \)
% in Lemma~\ref{lem:MRC},
% we have \( \tilde{q}(X) = q(Y) \leq \dim Y < \dim X = n\) by
% Lemma~\ref{lem:qYqX}, \eqref{lem:qYqX:2}.
% This is a contradiction. Hence, \( X \) is not uniruled.
% Then, \( X \) has only canonical singularities with \( K_{X} \sim_{\BQQ} 0 \) by
% Theorem~\ref{thm:nonuniruled}, and furthermore,
% \( X \) is abelian by Proposition~\ref{prop:split}, \eqref{prop:split:2},
% since \( q(X) = \tilde{q}(X) = n \).
% Thus, we have proved \eqref{cor:qXqY:1}.

\eqref{cor:qXqY:2}:
Let \( \hat{\tau} \colon \widehat{X} \to X \) be
a finite covering \'etale in codimension one from a
normal projective variety \( \widehat{X} \) and
\( \hat{f} \colon \widehat{X} \to \widehat{X} \) an endomorphism such that
\( \hat{\tau} \circ \hat{f} = f \circ \hat{\tau} \).
Then, \( \tau_{*} \colon \pi_{1}(\hat{\tau}^{-1}(X_{\reg})) \to \pi_{1}(X_{\reg}) \)
is injective and its image is a finite-index subgroup (cf.\ Remark~\ref{etale1}).
Since the natural inclusion \( X_{\reg} \injmap X \) induces a surjection
\( \pi_{1}(X_{\reg}) \to \pi_{1}(X) \), the image of
\( \tau_{*} \colon \pi_{1}(\widehat{X}) \to \pi_{1}(X) \) is
also a finite-index subgroup.
Thus, we may replace \( (X, f) \) with \( (\widehat{X}, \hat{f}) \).
Therefore, we can assume that \( q^{\natural}(X, f) = \tilde{q}(X) \).
Let \( \q \) be the positive integer defined by \( \q^{n} = \deg f \).

Assume first that \( X \) is not uniruled.
Then, \( X \) has only canonical singularities and \( K_{X} \sim_{\BQQ} 0 \) by
Theorem~\ref{thm:nonuniruled}.
By Proposition~\ref{prop:qXfqX}, we may assume that \( X = A \times S \)
and \( f = f_{A} \times f_{S} \) for
an abelian variety \( A \), a weak Calabi--Yau variety \( S \), and
polarized endomorphisms \( f_{A} \colon A \to A \) and \( f_{S} \colon S \to S \)
with \( \deg f_{A} = \q^{\dim A} \) and
\( \deg f_{S} = \q^{\dim S} \), respectively.
Since \( \dim S = n - \dim A = n - q^{\natural}(X, f) \), we infer that
\( S \) is a point by our assumption on Conjecture~\ref{conj:Qab}.
Thus, \( n = q^{\natural}(X, f) \), \( X = A\), and \( \pi_{1}(X) \) is
a free abelian group of rank \( 2n = 2q^{\natural}(X, f) \).
Therefore, the assertion \eqref{cor:qXqY:2}
is true when \( X \) is not uniruled.

Assume next that \( X \) is uniruled.
Let \( \sigma \colon W \to X \) and \( p \colon W \to Y \)
be as in Lemma~\ref{lem:MRC}. Then, we have a surjection
\( \pi_{1}(Y) \to \pi_{1}(X) \) by Lemma~\ref{lem:qYqX}, \eqref{lem:qYqX:3}.
In particular, if \( Y \) is a point, then \( \pi_{1}(X) \) is trivial.
Thus, we may assume that \( \dim Y > 0\). Then,
\( Y \) is not uniruled and it admits
a polarized endomorphism \( f_{Y} \colon Y \to Y\)
of degree \( \q^{\dim Y} > 1\) by Lemma~\ref{lem:MRC}.
We have
\( q^{\natural}(Y, f_{Y}) = q^{\circ}(Y, f_{Y}) = \tilde{q}(X) = q^{\natural}(X, f)\)
by Lemma~\ref{lem:qYqX}, \eqref{lem:qYqX:2}.
Since \( \dim Y - q^{\natural}(Y, f_{Y}) < n - q^{\natural}(X, f)  \),
we can apply the previous argument to the non-uniruled variety \( Y \).
Thus, \( \pi_{1}(Y) \) contains a finite-index subgroup
which is a finitely generated abelian
group of rank at most \( 2q^{\natural}(Y, f_{Y}) = 2q^{\natural}(X, f)\).
Hence, \( \pi_{1}(X) \) has the same property,
since we have the surjection \( \pi_{1}(Y) \to \pi_{1}(X) \).
Therefore, the assertion \eqref{cor:qXqY:2} has been proved.
\end{proof}

Now, we are ready to prove Theorem~\ref{Th0}, which is a consequence of
Theorems~\ref{thm:nonuniruled} and \ref{thm:solQab}, Proposition~\ref{prop:qXfqX},
and Lemmas~\ref{lem:MRC} and \ref{lem:qYqXPart}.

\begin{proof}[Proof of Theorem~\ref{Th0}]
If \( X \) is not uniruled, then it is proved in
Theorems~\ref{thm:nonuniruled} and \ref{thm:solQab}, and
Proposition~\ref{prop:qXfqX}.
Thus, we may assume that \( X \) is uniruled.
We apply Lemma~\ref{lem:MRC} to the polarized endomorphism \( f \colon X \to X \).
Let \( \sigma \colon W \to X \), \( p \colon W \to Y \),
\( f_{W} \colon W \to W\), and \( f_{Y} \colon Y \to Y\)
be the same objects as in Lemma~\ref{lem:MRC}.
Then \( Y \) has only canonical
singularities and \( K_{Y} \sim_{\BQQ} 0 \) by
Theorem~\ref{thm:nonuniruled}. Moreover, by
Proposition~\ref{prop:qXfqX},
there exist a finite surjective morphism \( A \times S \to Y \)
\'etale in codimension one from the direct product
\( A \times S \) for an abelian variety \( A \) and a weak
Calabi--Yau variety \( S \), and polarized endomorphisms \( f_{A} \),
\( f_{S} \) such that \( f_{A} \times f_{S} \colon A \times S \to A \times S \)
is compatible with \( f_{Y} \).
Here, if \( \dim S > 0 \), then \( \dim S \geq 4 \)
and \( S \) has a non-quotient singularity
by Theorem~\ref{thm:solQab}.

Let \( Z \) be the normalization of the fiber product of
\( p \colon W \to Y \) and \( A \times S \to Y \).
Let \( \varpi \colon Z \to A \times S \) be the morphism induced
from the second projection
and let \( Z \xrightarrow{\rho} V \xrightarrow{\tau} X \) be the Stein factorization of
\( Z \to W \xrightarrow{\sigma} X \).
Then, we have a commutative diagram:
\[ \begin{CD}
A \times S @<{\varpi}<< Z @>{\rho}>> V \\
@VVV @VVV @V{\tau}VV \\
Y @<{p}<< W @>{\sigma}>> \phantom{.}X.
\end{CD} \]
By Lemma~\ref{lem:qYqXPart}, the following hold:
\begin{itemize}
\item  \( Z \) is irreducible.

\item The fibration \( Z \to A \times S \) is equi-dimensional and is birational to
the maximal rationally connected fibration of a smooth model of \( Z \).

\item  The finite surjective morphisms
\( Z \to W \) and \( V \to X \) are \'etale in codimension one.

\item  There exist
polarized
endomorphisms \( f_{Z} \colon Z \to Z \) and
\( f_{V} \colon V \to V \) such that \( f_{V} \) is compatible with \( f \) and that
\( f_{Z} \) is compatible with \( f_{W} \), \( f_{V} \) and \( f_{A} \times f_{S} \).

\end{itemize}
Let \( \pi \colon V \ratmap A \times S\) be
the rational map \( \varpi \circ \rho^{-1}  \).
Then \( Z \) is just the normalization of the graph \( \Gamma_{\pi} \) of \( \pi \).
In particular, \( \Gamma_{\pi} \to A \times S\) is equi-dimensional.
Thus, the conditions required in Theorem~\ref{Th0} are all satisfied.
% Thus, all the assertions of Theorem~\ref{Th0} have been proved.
\end{proof}

The proof of Theorem~\ref{Th1.2new} uses the following:

\begin{lemma}\label{lem:P1bdl}
Let \( Z \) be the normalization of the graph of
\( \pi \colon V \ratmap A \times S \) in \emph{Theorem~\ref{Th0}}
and let \( \varpi \colon Z \to A \times S\) be
the induced equi-dimensional morphism.
Suppose that \( \dim S = 0\). Then \( \varpi \) is flat, and
any fiber of \( \varpi \) is irreducible, normal,
and rationally connected.
If \( \dim Z = \dim A + 1 \), then \( \varpi \) is
a holomorphic \( \BPP^{1}\)-bundle.
\end{lemma}

\begin{proof}
Let \( V \to X\), \( Z \to V\), \( Z \to A \), and \( A \to Y \)
be as in the proof of Theorem~\ref{Th0}, where \( S \) is a point.
We may assume that \( \dim A > 0 \). %%%added
Let \( Z_{1} \) be the fiber product of \( \varpi \colon Z \to A \)
and \( f_{A} \colon A \to A \).
Then the other endomorphism \( f_{Z} \colon Z \to Z\) induces
a commutative diagram
\begin{equation}\label{eq:cd}
\begin{CD}
Z @>{\psi}>> Z_{1} @>{p_{1}}>> Z \\
@V{\varpi}VV @V{p_{2}}VV @V{\varpi}VV \\
A @= A @>{f_{A}}>> \phantom{,}A,
\end{CD}\tag{*}
\end{equation}
where \( p_{1} \) and \( p_{2} \) denote the first and second
projections, and \( f_{Z} = p_{1} \circ \psi \).
Note that \( p_{1} \) is \'etale, since so is \( f_{A} \).
Thus, \( Z_{1} \) is also a normal projective variety
and \( \psi \) is a finite surjective morphism.

\vspace{1ex}

\emph{Step}~1.
We shall prove: \emph{If a subset \( \Sigma \subset A \) is
not Zariski dense
and \( f_{A}^{-1}(\Sigma) \subset \Sigma \), then \( \Sigma = \emptyset \)}.

We shall derive a contradiction by assuming \( \Sigma \ne \emptyset \).
First, we note that \( f_{A}^{-1}(\overline{\Sigma}) \subset \overline{\Sigma} \) for
the Zariski-closure \( \overline{\Sigma} \) of \( \Sigma \).
In fact, by assumption,  we have
\( f_{A}^{-1}(A \setminus \Sigma) \supset A \setminus \Sigma \supset
A \setminus \overline{\Sigma} \ne \emptyset\).
Thus, \( f_{A}(A \setminus \overline{\Sigma}) \) is
a Zariski-open subset contained in \( A \setminus \Sigma \),
since \( f_{A} \) is an open map.
Hence \( f_{A}(A \setminus \overline{\Sigma})  \subset A \setminus \overline{\Sigma} \),
and equivalently, \( f_{A}^{-1}(\overline{\Sigma}) \subset \overline{\Sigma}\).
% Assume the contrary that \( \Sigma \ne \emptyset \).
% Let \( \overline{\Sigma} \) be the Zariski closure of \( \Sigma \).
% We shall show that \( f_{A}^{-1}(\overline{\Sigma}) \subset \overline{\Sigma} \).
% We have \( f_{A}^{-1}(A \setminus \Sigma) \supset A \setminus \Sigma \supset
% A \setminus \overline{\Sigma} \ne \emptyset\) by assumption.
% Since \( f_{A} \) is an open map,
% \( f_{A}(A \setminus \overline{\Sigma}) \) is
% a Zariski open subset contained in \( A \setminus \Sigma \).
% Hence \( f_{A}(A \setminus \overline{\Sigma})  \subset A \setminus \overline{\Sigma} \),
% and equivalently, \( f_{A}^{-1}(\overline{\Sigma}) \subset \overline{\Sigma}\).
%
Therefore, replacing \( \Sigma \) with \( \overline{\Sigma} \),
we may assume that \( \Sigma \) is Zariski-closed.
There is a positive integer \( l \) such that \( f_{A}^{-l}(\Sigma) =
f_{A}^{-l-1}(\Sigma) \) by the Noetherian condition for Zariski-closed subsets.
Hence, \(f_{A}^{-1}(\Sigma) = \Sigma \).
Replacing \( f \) with a power \( f^{k} \), we
may assume that \( f_{A}^{-1} \) preserves every irreducible component
of \( \Sigma \). Thus, we may assume that \( \Sigma \) is
irreducible.
Let \( f_{\Sigma} \colon \Sigma \to \Sigma\)
be the polarized endomorphism of \( \Sigma \) induced from \( f_{A} \).
Then \( \deg f_{A} = \q^{\dim A} \) and
\( \deg f_{\Sigma} = \q^{\dim \Sigma} \) for some \( \q > 1 \)
by Lemma~\ref{lem:first}.
On the other hand, \( \deg f_{\Sigma} = \deg f_{A}  \),
since \( f_{A} \) is \'etale. Thus, \( \dim \Sigma = \dim A \).
This is a contradiction.

\vspace{1ex}

\emph{Step}~2. We shall prove: \emph{Any fiber of \( \varpi \) is irreducible}.

Let \( \Sigma \) be the set of points \( y \in A \) such that \(
\varpi^{-1}(y) \) is reducible. Then \( \varpi^{-1}(y') \) is
reducible for any \( y' \in f_{A}^{-1}(y) \), since \( \psi \) in the
diagram \eqref{eq:cd} is surjective. Thus, \( f_{A}^{-1}(\Sigma)
\subset \Sigma \). Since a general fiber of \( \varpi \) is irreducible,
we have \( \Sigma = \emptyset \) by
\emph{Step}~1.

\vspace{1ex}

\emph{Step}~3. We shall prove: \emph{\( \varpi \) is flat}.

Let \( L \) be an ample divisor on \( Z \) such that \( f_{Z}^{*}L
\sim \q L \). Since \( \SO_{Z_{1}} \) is a direct summand of
\(\psi_{*}\SO_{Z} \)
(cf.\ the first part of the proof of Lemma~\ref{lem:splitnormal} below), %%%added
we infer that
\( \varpi_{*}\SO_{Z}(f_{Z}^{*}L) \isom \varpi_{*}\SO_{Z}(\q L) \)
contains
\[ p_{2*}\SO_{Z_{1}}(p_{1}^{*}L) \isom
f_{A}^{*}\left(\varpi_{*}\SO_{Z}(L)\right) \]
as a direct summand. In particular, if \( \varpi_{*}\SO_{Z}(\q L) \) is
locally free at a point \( y \in A \), then so is \( \varpi_{*}\SO_{Z}(L)
\) at \( f_{A}(y) \).
Let \( U \) be the set of points \( y \in A \) such that \( \varpi \)
is flat along \( \varpi^{-1}(y) \). Then \( U \) is a Zariski open
dense subset. The argument above says that \( f_{A}(U) \subset U \),
since \( y \in U \) if and only if \( \varpi_{*}\SO_{Z}(mL) \) is
locally free at \( y \) for \( m \gg 0 \) (cf.\ \cite{EGA3II}, Proposition~7.9.14).
Thus, for the complement \( \Sigma \) of
\( U \) in \( A \), we have \( f_{A}^{-1}(\Sigma) \subset \Sigma \).
Then \( \Sigma = \emptyset \) by \emph{Step}~1, and hence \( \varpi \)
is flat.

\vspace{1ex}

\emph{Step}~4. We shall prove: \emph{Any fiber of \( \varpi \) is normal}.

Let \( \Sigma \) be the set of points \( y \in A \) such that
the fiber \( F_{y} := \varpi^{-1}(y) \) is not normal.
We fix a point \( y \in \Sigma \) and a non-normal point \( x \) of \( F_{y} \).
For a point \( y' \in f_{A}^{-1}(y) \), let \( x' \) be a point of \( F_{y'} \)
such that \( f(x') = x \).
Then, \( x_{1} := \psi(x') \) is a point of \( Z_{1} \) such that
\( \{x_{1}\} = p_{1}^{-1}(x) \cap p_{2}^{-1}(y')\).
Note that \( x_{1} \) is a non-normal point of \( p_{2}^{-1}(y') \),
since \( p_{2}^{-1}(y') \isom F_{y} \).

Assume that \( y' \not\in \Sigma \), i.e., \( F_{y'} \) is normal.
We have affine open neighborhoods \( U' \subset Z\) and \( U_{1} \subset Z_{1}\)
of \( x' \) and \( x_{1} \),
respectively, such that \( U' = \psi^{-1}(U_{1})\).
Thus, \( U' = \operatorname{Spec} R' \) and \( U_{1} = \operatorname{Spec} R_{1} \)
for finitely generated \( \BCC \)-algebras \( R' \) and \( R_{1} \)
such that \( R' \) and \( R_{1} \) are normal domains,
\( R_{1} \) is a subalgebra of \( R' \) and that
\( R' \) is a finite \( R_{1} \)-module.
Let \( I \) be the ideal of \( R_{1} \) defining the closed subscheme
\( U_{1} \cap p_{2}^{-1}(y') \isom U_{1} \times_{A} \{y'\} \).
Then, \( R'/IR' \) is normal, since
\( U' \cap F_{y'} \isom U' \times_{A} \{y'\}\) is normal.
Thus, \( R_{1}/I \) is normal by Lemma~\ref{lem:splitnormal} below.
Therefore, \( p_{2}^{-1}(y') \) is normal at \( x_{1} \). This is a contradiction.
Thus, \( y' \in \Sigma \).
Hence, \( f_{A}^{-1}(\Sigma) \subset \Sigma \).
Since a general fiber of \( \varpi \) is normal,
we have \( \Sigma = \emptyset \) by \emph{Step}~1,

\vspace{1ex}

\emph{Step}~5. We shall prove: \emph{Any fiber of \( \varpi \) is rationally connected}.

A general fiber of \( \varpi \) is rationally connected by the construction
of \( \varpi \) in the proof of Theorem~\ref{Th0}.
Let \( \Sigma \) be the set of points \( y \in A \) such that the fiber
\( F_{y} = \varpi^{-1}(y) \) is not rationally connected.
If \( F_{y'} \) is rationally connected
for a point \( y' \in f_{A}^{-1}(y) \), then \( F_{y} \) is rationally connected,
since \( F_{y'} \to F_{y} \) is surjective.
Thus, \( f_{A}^{-1}(\Sigma) \subset \Sigma \). Therefore, \( \Sigma = \emptyset \)
by \emph{Step}~1.

\vspace{1ex}

\emph{Step}~6. End of the proof:
% The remaining case: \( \dim Z/A = 1 \).

Finally, we consider the case where \( \dim Z/A = 1 \).
Then, \( \varpi \) is flat and any fiber of \( \varpi \) is \( \BPP^{1} \) by
\emph{Steps}~2--5. In particular, \( \varpi \) is smooth and is a
holomorphic \( \BPP^{1} \)-bundle. Thus, we are done.
\end{proof}

The lemma below on commutative algebra is used in \emph{Step}~4 of
the proof of Lemma~\ref{lem:P1bdl}:

\begin{lemma}\label{lem:splitnormal}
Let \( R_{0} \) and \( R_{1} \) be commutative algebras
finitely generated over a field \( k \) of characteristic zero.
Assume that \( R_{0} \) and \( R_{1} \) are normal integral domains,
\( R_{0} \) is a \( k \)-subalgebra of \( R_{1} \) and
that \( R_{1} \) is a finite \( R_{0} \)-module.
Let \( I \) be an ideal of \( R_{0} \) such that
\( R_{1}/IR_{1} \) is normal. Then, \( R_{0}/I \) is also normal.
\end{lemma}

\begin{proof}
Since the characteristic of \( k \) is zero,
\( R_{0} \) is a direct summand of the \( R_{0} \)-module \( R_{1} \).
This is shown as follows:
Let \( K_{i} \) be the field of fractions of \( R_{i} \) for \( i = 0 \), \( 1 \).
Then, \( K_{1} \) is a finite extension of \( K_{0} \).
Let \( t \colon K_{1} \to K_{0} \) be the trace map
of the extension \( K_{1}/K_{0} \): for \( a \in K_{1} \),
\( t(a) \) is the trace of the multiplication map
\( \mu(a) \colon K_{1} \to K_{1} \) by \( a \).
The composite \( K_{0} \to K_{1} \xrightarrow{t} K_{0} \)
with the canonical inclusion \( K_{0} \injmap K_{1} \)
is just the multiplication map by \( \deg (K_{1}/K_{0}) = \dim_{K_{0}} K_{1}\).
We have \( t(R_{1}) \subset R_{0} \),
since the eigenvalues of \( \mu(a) \) for
\( a \in R_{1} \) are integral over \( R_{0} \) and
since \( R_{0} \) is integrally closed in \( K_{0} \).
Moreover, the map \( R_{1} \to R_{0} \) induced by \( t \), is \( R_{0} \)-linear.
Thus, \( R_{0} \) is a direct summand of the \( R_{0} \)-module \( R_{1} \),
since \( \deg(K_{1}/K_{0}) \ne 0\) in \( k \).

Therefore, the natural homomorphism
\[ R_{0}/I \to R_{1} \otimes_{R_{0}} (R_{0}/I) \isom  R_{1}/IR_{1} \]
is injective
and \( R_{0}/I \) is regarded as a direct summand of \( R_{1}/IR_{1} \).
Since \( R_{1}/IR_{1} \) is an integral domain, so is \( R_{0}/I \).
Let \( \overline{R} \) be the normalization of \( R_{0}/I \).
Then, \( R_{0}/I \subset \overline{R} \subset R_{1}/IR_{1} \),
since \( R_{1}/IR_{1} \) is normal.
Thus, \( R_{0}/I \) is a direct summand of the \( R_{0}/I \)-module
\( \overline{R} \).
Let \( \overline{R} \to R_{0}/I \) be a projection to the direct summand
%
% Let \( \overline{R} \to R_{0}/I \) be a homomorphism of \( R_{0}/I \)-modules
% such that the composite \( R_{0}/I \subset \overline{R} \to R_{0}/I \)
% is the identity map.
and let \( M \) be the kernel of \( \overline{R} \to R_{0}/I \).
Then, \( M \otimes_{R/I_{0}} \overline{K} = 0\) for the field \( \overline{K} \)
of fractions
of \( R/I_{0} \).
Since \( \overline{R} \to \overline{R} \otimes_{R/I_{0}} \overline{K}
\isom \overline{K} \) is injective, we have \( M = 0 \).
Therefore, \( R/I_{0} \isom \overline{R}\), i.e., \( R/I_{0} \) is normal.
\end{proof}

A holomorphic \( \BPP^{1} \)-bundle is not necessarily associated with
a locally free sheaf of rank two. But we have the following result on
holomorphic \( \BPP^{1} \)-bundles over abelian varieties:

\begin{lemma}\label{lem:holP1bdle}
Let \( Z \to A \) be a holomorphic \( \BPP^{1} \)-bundle over an abelian variety \( A \).
For the multiplication map \( \nu_{2} \colon A \to A \) by \( 2 \),
let \( Z' \to A\) be the \( \BPP^{1} \)-bundle obtained by the pullback of \( Z \to A \)
by \( \nu_{2} \). Then, there exists a locally free sheaf \( \SE \) of rank two on \( A \)
such that \( Z' \isom \BPP_{A}(\SE) \) and \( \det \SE \isom \SO_{A} \).
\end{lemma}

\begin{proof}
We have an exact sequence
\[ 1 \to \boldsymbol{\mu}_{2, A} \to \operatorname{\underline{\mathit{SL}}}(2, \SO_{A})
\to \operatorname{\underline{\mathit{PGL}}}(2, \SO_{A}) \to 1\]
of sheaves of non-commutative groups on \( A \),
where \( \operatorname{\underline{\mathit{SL}}}(2, \SO_{A}) \)
(resp.\ \( \operatorname{\underline{\mathit{PGL}}}(2, \SO_{A}) \)) is the
sheaf of germs of \( \operatorname{SL}(2, \BCC) \)-valued
(resp.\ \( \operatorname{PGL}(2, \BCC) \)-valued)
holomorphic functions on \( A \), and
\( \boldsymbol{\mu}_{2, A} \isom (\BZZ/2\BZZ)_{A} \) is the constant sheaf
of \( \boldsymbol{\mu}_{2} = \{\pm 1\} \subset \BCC^{\star} \).
Since \( \boldsymbol{\mu}_{2} \) is the center of \( \operatorname{SL}(2, \BCC) \),
we have an associated exact sequence
\[ \check{\mathrm{H}}^{1}(A, \operatorname{\underline{\mathit{SL}}}(2, \SO_{A})) \to
\check{\mathrm{H}}^{1}(A, \operatorname{\underline{\mathit{PGL}}}(2, \SO_{A}))
\to \OH^{2}(A, \boldsymbol{\mu}_{2, A}) \]
of the $\check{\mathrm{C}}$ech cohomology sets.
The holomorphic \( \BPP^{1} \)-bundle \( Z/A \) is associated with an element
of \( \eta \) of
\( \check{\mathrm{H}}^{1}(A, \operatorname{\underline{\mathit{PGL}}}(2, \SO_{A})) \).
If the image \( \overline{\eta} \) in \( \OH^{2}(A, \boldsymbol{\mu}_{2, A}) \) is zero,
then \( Z \isom \BPP_{A}(\SE) \)
for a locally free sheaf \( \SE \) of rank two with \( \det \SE \isom \SO_{A} \)
which is associated with an element of
\( \check{\mathrm{H}}^{1}(A, \operatorname{\underline{\mathit{SL}}}(2, \SO_{A})) \).
Thus, it is enough to show that \( \overline{\eta} \)
is mapped to zero by
the homomorphism
\( \nu_{2}^{*} \colon \OH^{2}(A, \BZZ/2\BZZ) \to \OH^{2}(A, \BZZ/2\BZZ) \).
The pullback homomorphism
\( \nu_{2}^{*} \colon \OH^{1}(A, \BZZ) \to \OH^{1}(A, \BZZ) \) is just the multiplication map
by \( 2 \).
Since \( \OH^{3}(A, \BZZ) \isom \bigwedge^{3}\OH^{1}(A, \BZZ) \) is torsion free,
the natural sequence
\[ \OH^{2}(A, \BZZ) \xrightarrow{2\times} \OH^{2}(A, \BZZ) \to
\OH^{2}(A, \BZZ/2\BZZ) \to 0\]
is exact; equivalently,
\( \OH^{2}(A, \BZZ) \otimes \BZZ/2\BZZ \isom \OH^{2}(A, \BZZ/2\BZZ) \).
Hence, \( \nu_{2}^{*} \colon \OH^{2}(A, \BZZ/2\BZZ) \to \OH^{2}(A, \BZZ/2\BZZ) \) is
zero; in particular, \( \overline{\eta} \) is mapped to zero by \( \nu_{2}^{*} \).
Thus, we are done.
\end{proof}

Now we are ready to prove Theorem~\ref{Th1.2new}, which follows
essentially from Theorems~\ref{Th0}, \ref{thm:nonuniruled}, and \ref{thm:solQab},
Lemmas~\ref{lem:MRC}, \ref{lem:qYqX}, and \ref{lem:P1bdl},
and Corollary~\ref{cor:qXqY}.

\begin{proof}[Proof of Theorem~\ref{Th1.2new}]
We may assume that \( n = \dim X > 0 \).
We have \( \kappa(X) \leq 0 \) by Theorem~\ref{Th0}.
The inequality \( q^{\natural}(X, f) \leq n \) and
the assertion \eqref{Th1.2new:n} are proved in
Corollary~\ref{cor:qXqY}, \eqref{cor:qXqY:1}.

We shall prove \eqref{Th1.2new:0}: Assume that
\( q^{\natural}(X, f) = 0 \) and Conjecture~\ref{conj:Qab} is true for varieties
of dimension at most \( n = \dim X \).
If \( X \) is not uniruled, then \( X \) is Q-abelian by
the conjecture;
thus
\( n = q^{\natural}(X, f) = 0 \) by Corollary~\ref{cor:qXqY}, \eqref{cor:qXqY:1}.
Hence \( X \) is uniruled.
Let \( \sigma \colon W \to X \) and \( p \colon W \to Y \) be as in
Lemma~\ref{lem:MRC}.
Since \( \dim Y \leq n \),
\( Y \) is Q-abelian by
the conjecture.
On the other hand, \( q^{\circ}(Y) = 0 \) by
Lemma~\ref{lem:qYqX}, \eqref{lem:qYqX:2}. Thus, \( Y \) is a point.
This means that \( X \) is rationally connected.

Next, we shall prove \eqref{Th1.2new:other} and \eqref{Th1.2new:n-1}:
Suppose that
Conjecture~\ref{conj:Qab} is true for varieties of dimension
at most \( n - q^{\natural}(X, f) \).
Let \( \tau \colon V \to X\), \( \rho \colon Z \to V\),
\( \varpi \colon Z \to A \times S\), and \( A \times S \to Y \) be as in the proof of
Theorem~\ref{Th0}.
By the proof of Corollary~\ref{cor:qXqY}, \eqref{cor:qXqY:2},
we infer that \( S \) is a point and \( q^{\natural}(X, f) = \dim A \).
Then \( \varpi \colon Z \to A \) is a flat morphism whose fibers are
all irreducible, normal, and rationally connected
by Lemma~\ref{lem:P1bdl}.
We have polarized endomorphisms \( f_{V} \colon V \to V \), \( f_{Z} \colon Z \to Z \),
and \( f_{A} \colon A \to A \) satisfying the compatibility conditions
in \eqref{Th1.2new:other} by the proof of Theorem~\ref{Th0}.
Thus the assertion \eqref{Th1.2new:other} follows.

Suppose that \( q^{\natural}(X, f) = n - 1 \).
Then \( \varpi \colon Z \to A \) is
a holomorphic
\( \BPP^{1} \)-bundle by Lemma~\ref{lem:P1bdl}.
Assume that \( \rho \colon Z \to V \) is an isomorphism.
Then, this corresponds to the case \eqref{Th1.2new:n-1:a} in a weak sense.
However, by replacing \( V \) by a finite \'etale covering, we can prove
that \( V \) is a \( \BPP^{1} \)-bundle associated
with a locally free sheaf of rank two, as follows.
By Lemma~\ref{lem:holP1bdle}, the fiber product \( Z \times_{A, \nu_{2}} A \)
is a \( \BPP^{1} \)-bundle
associated with a locally free sheaf of rank two for the multiplication map
\( \nu_{2} \colon A \to A \) by \( 2 \) with respect
to a certain group structure of \( A \).
There is an endomorphism \( f'_{A} \colon A \to A \) such that
\( \nu_{2} \circ f'_{A} = f_{A} \circ \nu_2 \), by
\cite{NZ}, Lemma~4.9.
Thus, we may replace \( Z \isom V \) with the \'etale covering
\( Z \times_{A, \nu_{2}} A \) by Lemma~\ref{lem:qYqXPart}
(cf.\ the proof of Theorem~\ref{Th0}).

Assume next that \( \rho \colon Z \to V \) is not isomorphic.
Let \( E \subset Z \) be the
exceptional locus. Then \( f_{Z}^{-1}(E) = E \),
since \( f_{Z} \) is compatible with \( f_{V} \).
Moreover, \( f_{A}^{-1}(\varpi(E)) = \varpi(E) \),
since \( \psi \colon Z \to Z_{1} \) is surjective.
Thus, \( \varpi(E) = A \) by \emph{Step}~1 in the proof of
Lemma~\ref{lem:P1bdl}.
Let \( \Sigma \subset A\) be the set of points \( y \in A \) such
that \( \varpi^{-1}(y) \subset E \).
Then \( f_{A}^{-1}(\Sigma) \subset \Sigma \).
Hence, \( \Sigma = \emptyset \) by
\emph{Step}~1 in the proof of Lemma~\ref{lem:P1bdl}.
Therefore,
\( \varpi|_{E} \colon E \to A \) is a finite surjective morphism.
% In order to prove that the condition \eqref{Th1.2new:n-1:b}
% is satisfied in this case,
It is enough to show that \( E \) is a
%rational section of \( \varpi \), which is equivalent to
%that \( E \) is a section of \( \varpi \),
section of \( \varpi \),
which is equivalent to
that \( \varpi|_{E} \colon E \to A \) is bijective,
since \( A \) is normal.
If \( E \) is a section of \( \varpi \), then \( \varpi \) is a \( \BPP^{1} \)-bundle
associated with the locally free sheaf \( \varpi_{*}\SO_{Z}(E) \) of rank two.

Let \( P \in A \) be an arbitrary point. Then, there exists a positive-dimensional
fiber \( \Gamma \) of \( E \subset Z \to V \) such that \( P \in \varpi(\Gamma) \).
Let \( C \) be the normalization of an irreducible curve in \( \varpi(\Gamma) \)
passing through \( P \) and let \( \nu \colon C \to A \) be the induced finite morphism.
Then, \( Z \times_{A} C \) is a \( \BPP^{1} \)-bundle over \( C \).
It suffices to prove that the support of
\( E \times_{A} C \) is a section of the \( \BPP^{1} \)-bundle.
Note that \( Z \times_{A} C \to Z \to V \) is generically injective and
it contracts any irreducible component \( \gamma \) of \( E \times_{A} C \)
to a point. Since \( Z \times_{A} C \) is a \( \BPP^{1} \)-bundle over \( C \),
the irreducible component \( \gamma \) is a unique curve of
\( Z \times_{A} C \) with negative self-intersection number and
it is a section of the \( \BPP^{1} \)-bundle.
Hence the support of \( E \times_{A} C \) is just the section \( \gamma \).
Therefore, \( E \) is a section of \( \varpi \), and the condition
\eqref{Th1.2new:n-1:b} is satisfied.
Thus, we are done.
\end{proof}

%\section{References}


\begin{thebibliography}{99}

%\bibitem{AF}
%M.~Abe and M.~Furushima, On non-normal del~Pezzo
%surfaces, Math.\ Nachr.\ \textbf{260} (2003), 3--13.

\bibitem{Am97}
E.~Amerik,
Maps onto certain Fano threefolds, Doc.\ Math. \textbf{2} (1997), 195--211.

\bibitem{Am}
E.~Amerik,
On endomorphisms of projective bundles,
Manuscripta Math.\  \textbf{111}  (2003), 17--28.

\bibitem{ARV}
E.~Amerik, M.~Rovinsky and A.~Van de Ven,
A boundedness theorem for morphisms between threefolds,
Ann.\ Inst.\ Fourier (Grenoble) \textbf{49} (1999), 405--415. %no. 2

\bibitem{Be1} A.~Beauville, Vari\'et\'e K\"ahlerinnes
dont la premiere classe de Chern est nulle,
J.\ Diff.\ Geom.\ \textbf{18} (1983), 755--782.

%\bibitem{Be2} A.~Beauville,
%Endomorphisms of hypersurfaces and other manifolds,
%Internat.\ Math.\ Res.\ Notices 2001, no.~1, 53--58.


\bibitem{Bi} G.~Birkhoff,
Linear transformations with invariant cones,
Amer.\ Math.\ Monthly \textbf{74} (1967), 274--276.

\bibitem{BDPP}
S.~Boucksom, J.-P.~Demailly, M.~Paun and T.~Peternell,
The pseudo-effective cone of a compact K\"ahler manifold
and varieties of negative Kodaira dimension,
preprint 2004 (math.AG/0405285).

%\bibitem{BCS}
%J.-V.~Briend, S.~Cantat and M.~Shishikura,
%Linearity of the exceptional set for maps of \(\BPP_{k}(\BCC)\),
%Math.\ Ann.\  \textbf{330}  (2004), 39--43.

\bibitem{Cp92}
F.~Campana, Connexit\'e rationnelle des vari\'et\'es de Fano, Ann.\
Sci.\ \'Ecole Norm.\ Sup.\ (4) \textbf{25} (1992), 539--545.


\bibitem{Cp04}
F.~Campana,
Orbifoldes \`a premi\`ere classe de Chern nulle,
\emph{The Fano Conference}, pp.~339--351,
Univ.\  Torino, 2004.

\bibitem{Ct}
S.~Cantat, Endomorphismes de vari\'et\'es homogenes,
Enseign.\ Math.\ \textbf{49} (2003), 237--262.

%\bibitem{CLN}
%D.~Cerveau and A.~Lins Neto, Hypersurfaces exceptionnelles
%des endomorphismes de \(CP(n)\),
%Bol.\ Soc.\ Brasil.\ Mat.\ (N.S.)  \textbf{31}  (2000),
%no.~2, 155--161.

\bibitem{Cu} S.~D.~Cutkosky,
Zariski decomposition of divisors on algebraic varieties,
Duke Math.\ J., \textbf{53} (1986), 149--156.

%\bibitem{De} P.~Deligne,
%Le groupe fondamental du compl\'ement d'une courbe plane n'ayant que des points
%doubles ordinaires est ab\'elien (d'apr\`es W.~Fulton),
%\emph{Bourbaki Seminar, Vol.\ 1979/80}, pp.~1--10,
%Lecture Notes in Math. \textbf{842}, Springer, 1981.

\bibitem{Fa} N.~Fakhruddin,
Questions on self maps of algebraic varieties,
J.\ Ramanujan Math.\ Soc.\ \textbf{18} (2) (2003), 109--122.

%\bibitem{FS} J.~E.~Fornaess and N.~Sibony,
%Complex dynamics in higher dimension. I,
%\emph{Complex analytic methods in dynamical systems}
%(Rio de Janeiro, 1992), pp.~201--231,
%Ast\'erisque \textbf{222}, Soc.\ Math.\ France, 1994.

\bibitem{Fm} Y.~Fujimoto,
Endomorphisms of smooth projective
\( 3 \)-folds with nonnegative Kodaira dimension,
Publ.\ Res.\ Inst.\ Math.\ Sci. Kyoto Univ.\
\textbf{38} (2002), 33--92.

\bibitem{FN1} Y.~Fujimoto and N.~Nakayama,
Compact complex surfaces admitting non-trivial surjective
endomorphisms, Tohoku Math.\ J.\ \textbf{57} (2005), 395--426.

\bibitem{FN2} Y.~Fujimoto and N.~Nakayama,
Endomorphisms of smooth projective
\(3\)-folds with non-negative Kodaira dimension, II,
J.\ Math.\ Kyoto Univ.\ \textbf{47} (2007), 79--114.

\bibitem{Ft} T.~Fujita,
Zariski decomposition and canonical rings of elliptic
threefolds,
J.\ Math.\ Soc.\ Japan \textbf{38} (1986), 19--37.

%\bibitem{Ftn} W.~Fulton,
%On the fundamental group of the complement of a node curve,
%Ann.\ of Math.\ \textbf{111} (1980), 407--409.

\bibitem{EGA3II} A.~Grothendieck,
\'El\'ements de g\'eom\'etrie alg\'ebrique: III.\
\'Etude cohomologique des faisceaux coh\'erents, II,
Publ.\ Math.\ Inst.\ Hautes \'Etudes Sci.\ \textbf{17} (1963), 5--91.

\bibitem{GHS}
T.~Graber, J.~Harris and J.~Starr,
Families of rationally connected varieties,
J.\ Amer.\ Math.\ Soc.\ \textbf{16} (2003), 57--67.

%\bibitem{Gu} R.~V.~Gurjar,
%On ramification of self-maps of \( \BPP^{2}\),
%J.\ Algebra \textbf{259} (2003), 191--200.

\bibitem{Hm} M.~Hanamura,
Structure of birational automorphism groups, I: non-uniruled varieties,
Invent.\ Math.\ \textbf{93} (1988), 383--403.

%\bibitem{HW} F.~Hidaka and K.-i.~Watanabe,
%Normal Gorenstein surfaces with ample anti-canonical divisor,
%Tokyo J.\ Math.\ \textbf{4} (1981), 319--330.

\bibitem{Ho} C.~Horst,
Compact varieties of surjective holomorphic endomorphisms,
Math.\ Z.\ \textbf{190}  (1985), 499--504.

%\bibitem{HL} H.~Hamm and L\^e D\~ung Trang,
%Un th\'eor\`eme de Zariski du type de Lefschetz,
%Ann.\ Sci.\ \'Ec.\ Norm.\ Sup.\ \textbf{6} (1973), 317--366.

\bibitem{HM} J.-M.~Hwang and N.~Mok,
Finite morphisms onto Fano manifolds of Picard number \( 1 \)
which have rational curves with trivial normal bundles,
J.\ Alg.\ Geom.\ \textbf{12}  (2003), 627--651.

%\bibitem{HN} J.-M.~Hwang and N.~Nakayama,
%On endomorphisms of Fano manifolds of Picard number one,
%preprint RIMS-\textbf{1628}, Kyoto Univ., 2008.

%\bibitem{IitakaGTM} S.~Iitaka,
%\emph{Algebraic Geometry, An Introduction to Birational Geometry of
%Algebraic Varieties}, Grad.\ Texts in Math.\ \textbf{76},
%Springer, 1982.

\bibitem{KaAV} Y.~Kawamata,
Characterization of abelian varieties,
Comp.\ Math.\ \textbf{43} (1981), 253--276.

\bibitem{Ka} Y.~Kawamata,
The cone of curves of algebraic varieties,
Ann.\ Math., \textbf{119} (1984), 603--633.

\bibitem{Ka85} Y.~Kawamata,
Minimal models and the Kodaira dimension of algebraic fiber spaces,
J.\ Reine Angew.\ Math.\ \textbf{363} (1985), 1--46.

\bibitem{Ka87} Y.~Kawamata,
The Zariski decomposition of log-canonical divisors,
\emph{Algebraic Geometry, Bowdoin, 1985} (S.~Bloch ed.),
Proc.\ Symp.\ Pure Math., \textbf{46}, pp.~425--433,
Amer.\ Math.\ Soc., 1987.

\bibitem{KMM} Y.~Kawamata, K.~Matsuda and K.~Matsuki,
Introduction to the minimal model problem,
\emph{Algebraic geometry, Sendai, 1985} (T.~Oda ed.), pp.~283--360,
Adv.\ Stud.\ Pure Math., \textbf{10}, Kinokuniya and North-Holland, 1987.

%\bibitem{Kl} S.~L.~Kleiman,
%The enumerative theory of singularities,
%\emph{Real and complex singularities, Oslo 1976},
%Proc.\ of the Ninth Nordic Summer School/NAVF Sympos.\ in Math.,
%Oslo, 1976 (P.~Holm ed.),
%pp.~297--396, Sijthoff and Noordhoff, 1977.

\bibitem{Ko86} J.~Koll\'ar,
Higher direct images of dualizing sheaves, II,
Ann.\ of Math.\ \textbf{124} (1986), 171--202.

\bibitem{Ko93} J.~Koll\'ar,
Shararevich maps and plurigenera of algebraic varieties,
Invent.\ Math.\ \textbf{113} (1993), 177--215.

\bibitem{KM} J.~Koll\'ar and S.~Mori,
\emph{Birational geometry of algebraic varieties},
Cambridge Tracts in Math., \textbf{134},
Cambridge Univ.\ Press, 1998.

\bibitem{KoMM} J.~Koll\'ar, Y.~Miyaoka and S.~Mori,
Rational connected varieties,
J.\ Alg.\ Geom.\ \textbf{1} (1992), 429--448.

%\bibitem{L} R.~Lazarsfeld,
%Some applications of the theory of positive vector bundles,
%\emph{Complete intersections} (Acireale, 1983), pp.~29--61,
%Lecture Notes in Math. \textbf{1092}, Springer, 1984.

\bibitem{LBook2} R.~Lazarsfeld,
\emph{Positivity in algebraic geometry II},
Positivity for vector bundles, and multiplier ideals,
Ergeb.\ der Math.\ und ihr.\ Grenz.\ 3.\ Folge.\
\textbf{49}, Springer, 2004.

\bibitem{MM} Y.~Miyaoka and S.~Mori,
A numerical criterion for uniruledness,
Ann.\ of Math.\ \textbf{124} (1986), 65--69.

%\bibitem{Mo} S.~Mori,
%Threefolds whose canonical bundles are not numerically effective,
%Ann.\ of Math. (2) \ \textbf{116} (1982), no.~1, 133--176.

\bibitem{Mw} A.~Moriwaki,
Semi-ampleness of the numerically effective part of
Zariski-decomposition,
J.\ Math.\ Kyoto Univ.\ \textbf{26} (1986), 461--481.

\bibitem{Mm} D.~Mumford,
\emph{Abelian Varieties},
Tata Inst.\ of Fund.\ Research, Oxford Univ.\ Press, 1970.

\bibitem{Ny99} N.~Nakayama,
Compact K\"ahler manifolds whose universal
covering spaces are biholomorphic to \(\BCC^n\),
\emph{preprint} RIMS-\textbf{1230}, Kyoto Univ., 1999;
currently in the midst of revision.

\bibitem{Ny02} N.~Nakayama,
Ruled surfaces with non-trivial surjective endomorphisms,
Kyushu J.\ Math.\ \textbf{56} (2002), 433--446.

\bibitem{Ny04} N.~Nakayama,
\emph{Zariski-decomposition and abundance},
MSJ Memoirs Vol.~\textbf{14}, Math.\ Soc.\ Japan, 2004.

\bibitem{IntS}
N.~Nakayama, Intersection sheaves over normal schemes,
preprint RIMS-\textbf{1614}, Kyoto Univ., 2007.

\bibitem{ENS}
N.~Nakayama, On complex normal projective surfaces
admitting non-isomorphic surjective endomorphisms,
in preparation.

\bibitem{NZ}
N.~Nakayama and D.-Q.~Zhang,
Building blocks of \'etale endomorphisms of
complex projective manifolds,
Proc. \ London Math. \ Soc. \ (to appear), also: \
arXiv:\textbf{0903.3729}.
%preprint RIMS-\textbf{1577}, Kyoto Univ., 2007.


\bibitem{PS} K.~H.~Paranjape and V.~Srinivas,
Self-maps of homogeneous spaces,
Invent.\ Math.\ \textbf{98}  (1989), 425--444.

\bibitem{Re80} M.~Reid,
Canonical \( 3 \)-folds,
\emph{Journ\'ees de G\'eometrie Alg\'ebrique d'Angers}
(ed.\ A.~Beauville),
pp.~273--310, Sijthoff and Noordhoff, 1980.

%\bibitem{Re} M.~Reid,
%Nonnormal del Pezzo surfaces,
%Publ.\ Res.\ Inst.\ Math.\ Sci.\ Kyoto Univ.\
%\textbf{30} (1994), 695--727.

\bibitem{SW} N.~I.~Shepherd-Barron and P.~M.~H.~Wilson,
Singular threefolds with numerically trivial first
and second Chern classes,
J.\ Alg.\ Geom.\ \textbf{3} (1994), 265--281.

\bibitem{Ty} S.~Takayama,
Local simple connectedness of resolutions of log-terminal singularities,
Internat.\ J.\ Math.\ \textbf{14} (2003), 825--836.

\bibitem{Yau} S.~T.~Yau,
On the Ricci curvature of a compact K\"ahler manifold and
the complex Monge--Amp\`ere equations, I,
Comm.\ Pure and Appl.\ Math.\ \textbf{31} (1978), 339--411.

%\bibitem{Zar29} O.~Zariski,
%On the problem of existence of algebraic functions of two variables possessing
%a given branch curve,
%Amer.\ J.\  Math.\ \textbf{51} (1929), 305--328.

%\bibitem{Zar37} O.~Zariski,
%A theorem on the Poincar\'e group of an algebraic hypersurface,
%Ann.\ of Math.\ \textbf{38} (1937), 131--141.

%\bibitem{Zar39} O.~Zariski,
%Some results in the arithmetic theory of algebraic varieties,
%Amer.\ J.\ Math.\ \textbf{61} (1939), 249--294.

%\bibitem{ZarSLN} O.~Zariski,
%\emph{An Introduction to the Theory of Algebraic Surfaces},
%Lecture Notes in Math.\ \textbf{83}, Springer-Verlag, 1969.

%\bibitem{ZarS} O.~Zariski,
%\emph{Algebraic Surfaces},
%With appendices by S.~S.~Abhyankar, J.~Lipman, and D.~Mumford,
%Reprint of the second (1971) edition,
%Classics in Mathematics, Springer-Verlag, 1995.

\bibitem{Zs} S.-W.~Zhang,
Distributions in algebraic dynamics,
\emph{Surveys in Differential Geometry}, vol.~\textbf{10},
pp.~381--430, Int.\ Press, 2006.

\end{thebibliography}
\end{document}